\documentclass[12pt]{amsart} 
\RequirePackage{mathtools}
\RequirePackage{amsopn}
\RequirePackage{amsfonts}
\RequirePackage{paralist}
\RequirePackage{amssymb}
\RequirePackage{amsthm}
\RequirePackage{mathrsfs}

\usepackage[alphabetic]{amsrefs}
\renewcommand\MR[1]{\relax} % Must follow amsrefs
\usepackage{tikz}
\usetikzlibrary{cd,decorations.pathmorphing}
\usepackage{mathrsfs}

\newtheorem{thm}{Theorem}[section]
\numberwithin{equation}{section}

\newtheorem{cor}[thm]{Corollary}
\newtheorem{lemma}[thm]{Lemma}
\newtheorem{prop}[thm]{Proposition}

\theoremstyle{definition}
\newtheorem{definition}[thm]{Definition}

\theoremstyle{remark}
\newtheorem{remark}[thm]{Remark}
\newtheorem{example}[thm]{Example}

\newtheorem{mycomment}[thm]{Comment}
{\end{mycomment}\endgroup}
\def\mathcs{C^{*}}
\newcommand{\cs}{\ensuremath{\mathcs}}
\DeclareMathSymbol{\rtimes}{\mathbin}{AMSb}{"6F}

\def\R{\mathbf{R}}

\def\T{\mathbf{T}}
\def\Z{\mathbf{Z}}

\def\H{\mathcal{H}}
\def\L{\mathcal{L}}
\def\Q{\mathbf{Q}}

\DeclareMathOperator{\id}{id}
\DeclareMathOperator{\ext}{ext}
\def\set#1{\{\,#1\,\}}
\newcommand\sset[1]{\{#1\}}

\def\restr#1{|_{{#1}}}
\makeatletter
\def\labelenumi{\textnormal{(\@alph\c@enumi)}}
\def\theenumi{\@alph \c@enumi}
\def\labelenumii{\textnormal{(\@roman\c@enumii)}}
\def\theenumii{\@roman \c@enumii}
\newcount\charno
\def\alphapart#1{\charno=96
\advance\charno by#1\char\charno}

\makeatother

\def\<{\langle}
\def\>{\rangle}
\let\ipscriptstyle=\scriptscriptstyle
\def\lipsqueeze{{\mskip -3.0mu}}
\def\ripsqueeze{{\mskip -3.0mu}}
\def\ipcomma{\nobreak\mathrel{,}\nobreak}
\newbox\ipstrutbox
\setbox\ipstrutbox=\hbox{\vrule height8.5pt% depth 3.5pt
width 0pt}
\def\ipstrut{\copy\ipstrutbox}
\def\lip#1<#2,#3>{\mathopen{\relax_{\ipstrut\ipscriptstyle{
#1}}\lipsqueeze
\langle} #2\ipcomma #3 \rangle}
\def\blip#1<#2,#3>{\mathopen{\relax_{\ipstrut
\ipscriptstyle{ #1}}\lipsqueeze\bigl\langle} #2\ipcomma #3 \bigr\rangle}
\def\rip#1<#2,#3>{\langle #2\ipcomma #3
\rangle_{\ripsqueeze\ipstrut\ipscriptstyle{#1}}}
\def\brip#1<#2,#3>{\bigl\langle #2\ipcomma #3
\bigr\rangle_{\ripsqueeze\ipstrut\ipscriptstyle{#1}}}
\def\angsqueeze{\mskip -6mu}
\def\smangsqueeze{\mskip -3.7mu}
\def\trip#1<#2,#3>{\langle\smangsqueeze\langle #2\ipcomma #3
\rangle\smangsqueeze\rangle_{\ripsqueeze\ipstrut\ipscriptstyle{#1}}}
\def\btrip#1<#2,#3>{\bigl\langle\angsqueeze\bigl\langle #2\ipcomma
#3
\bigr\rangle
\angsqueeze\bigr\rangle_{\ripsqueeze\ipstrut\ipscriptstyle{#1}}}
\def\tlip#1<#2,#3>{\mathopen{\relax_{\ipstrut\ipscriptstyle{
#1}}\lipsqueeze \langle\smangsqueeze\langle} #2\ipcomma #3
\rangle\smangsqueeze\rangle}
\def\btlip#1<#2,#3>{\mathopen{\relax_{\ipstrut\ipscriptstyle{
#1}}\lipsqueeze
\bigl\langle\angsqueeze\bigl\langle} #2\ipcomma #3
\bigr\rangle\angsqueeze\bigr\rangle}

\def\ip(#1|#2){(#1\mid #2)}
\def\bip(#1|#2){\bigl(#1 \mid #2\bigr)}
\def\Bip(#1|#2){\Bigl( #1 \bigm| #2 \Bigr)}
\expandafter\ifx\csname BibSpec\endcsname\relax\else
\BibSpec{collection.article}{%
    +{}  {\PrintAuthors}                {author}
    +{,} { \textit}                     {title}
    +{.} { }                            {part}
    +{:} { \textit}                     {subtitle}
    +{,} { \PrintContributions}         {contribution}
    +{,} { \PrintConference}            {conference}
    +{}  {\PrintBook}                   {book}
    +{,} { }                            {booktitle}
    +{,} { }                            {series}
    +{,} { \voltext}                    {volume}
    +{,} { }                            {publisher}
    +{,} { }                            {organization}
    +{,} { }                            {address}
    +{,} { \PrintDateB}                 {date}
    +{,} { pp.~}                        {pages}
    +{,} { }                            {status}
    +{,} { \PrintDOI}                   {doi}
    +{,} { available at \eprint}        {eprint}
    +{}  { \parenthesize}               {language}
    +{}  { \PrintTranslation}           {translation}
    +{;} { \PrintReprint}               {reprint}
    +{.} { }                            {note}
    +{.} {}                             {transition}
}
\BibSpec{article}{%
    +{}  {\PrintAuthors}                {author}
    +{,} { \textit}                     {title}
    +{.} { }                            {part}
    +{:} { \textit}                     {subtitle}
    +{,} { \PrintContributions}         {contribution}
    +{.} { \PrintPartials}              {partial}
    +{,} { }                            {journal}
    +{}  { \textbf}                     {volume}
    +{}  { \PrintDatePV}                {date}
    +{,} { \eprintpages}                {pages}
    +{,} { }                            {status}
    +{,} { \PrintDOI}                   {doi}
    +{,} { available at \eprint}        {eprint}
    +{}  { \parenthesize}               {language}
    +{}  { \PrintTranslation}           {translation}
    +{;} { \PrintReprint}               {reprint}
    +{.} { }                            {note}
    +{.} {}                             {transition}
}
\BibSpec{book}{%
    +{}  {\PrintPrimary}                {transition}
    +{,} { \textit}                     {title}
    +{.} { }                            {part}
    +{:} { \textit}                     {subtitle}
    +{,} { \PrintEdition}               {edition}
    +{}  { \PrintEditorsB}              {editor}
    +{,} { \PrintTranslatorsC}          {translator}
    +{,} { \PrintContributions}         {contribution}
    +{,} { }                            {series}
    +{,} { \voltext}                    {volume}
    +{,} { }                            {publisher}
    +{,} { }                            {organization}
    +{,} { }                            {address}
    +{,} { pp.~}                        {pages}
    +{,} { \PrintDateB}                 {date}
    +{,} { }                            {status}
    +{}  { \parenthesize}               {language}
    +{}  { \PrintTranslation}           {translation}
    +{;} { \PrintReprint}               {reprint}
    +{.} { }                            {note}
    +{.} {}                             {transition}
}
\fi
\evensidemargin=\oddsidemargin
\mathtoolsset{showonlyrefs} % Show only labels that are referenced
\newcommand\U{\mathscr{U}}
\newcommand\V{\mathscr{V}}
\newcommand\sheaffont{\mathscr}
\newcommand\uG{\sheaffont{G}}
\newcommand\uT{\sheaffont{S}}

\newcommand\A{\mathcal A}
\newcommand\hA{\hat \A}
\newcommand{\B}{\mathcal B}
\newcommand{\D}{\mathcal D}
\newcommand{\G}{\mathcal G}

\newcommand\Sigmaw{\widetilde\Sigma}

\newcommand\go{\Sigma^{(0)}}
\newcommand\goo{\G^{(0)}}
\newcommand\hp{\hat p}

\newcommand\cc{C_{c}}

\let\phi\varphi

\newcommand\phiw{\hat\phi} % Used to be widehat

\def\SD(#1,#2){{#1}\triangleleft{#2}}

\newcommand\Tg{T_{\G}}
\newcommand\Tga{\Tg(\A)}

\newcommand{\dotk}{\dot{k}}

\newcommand\Twist{\Lambda}

\newcommand\fchi{f^{\chi}} % Bad notation?
\newcommand\fchiss{\fchi_{*}(\Sigma)}

\newcommand\Sigmah{\Sigmaw}

\newcommand{\hTwistOmega}{\widehat{\Twist}_\Omega}

\begin{document}
\begin{abstract}
  We analyse extensions $\Sigma$ of groupoids $\G$ by bundles $\A$ of
  abelian groups. We describe a pushout construction for such
  extensions, and use it to describe the extension group of a given
  groupoid $\G$ by a given bundle $\A$. There is a natural action
  of $\Sigma$ on the dual of $\A$, yielding a corresponding
  transformation groupoid. The pushout of this transformation groupoid
  by the natural map from the fibre product of $\A$ with its dual to
  the Cartesian product of the dual with the circle is a twist over the
 transformation groupoid resulting from the action of $\G$ on the dual
 of $\A$. We prove that the full 
  $C^*$-algebra of this twist is isomorphic to the full $C^*$-algebra
  of $\Sigma$, and that this isomorphism descends to an isomorphism of
  reduced algebras. We give a number of examples and applications.
\end{abstract}

\title{Pushouts of extensions of groupoids by bundles of abelian
  groups}

\date{10 July 2021}

\author[M. Ionescu]{Marius Ionescu}
\address{Department of Mathematics\\  United States Naval Academy\\
  Annapolis, MD 21402 USA}
\email{ionescu@usna.edu}

\author[A. Kumjian]{Alex Kumjian}
\address{Department of Mathematics \\ University of Nevada\\ Reno NV
  89557 USA}
\email{alex@unr.edu}

\author{Jean N. Renault}
\address{Institut Denis Poisson (UMR 7013)\\
  Universit\'e d'Orl\'eans et CNRS \\ 45067
  Orl\'eans Cedex 2, France}
\email{jean.renault@univ-orleans.fr}

\author[A. Sims]{Aidan Sims}
\address{School of Mathematics and Applied Statistics\\ University of
  Wollongong\\ NSW 2522, Australia}
\email{asims@uow.edu.au}

\author[D. P. Williams]{Dana P. Williams}
\address{Department of Mathematics\\ Dartmouth College \\ Hanover, NH
  03755-3551 USA}
\email{dana.williams@Dartmouth.edu}

\dedicatory{We respectfully dedicate this paper to the memory of
  Vaughan Jones: Extraordinary mathematician, proud New Zealander, and
  gracious colleague.}

\thanks{This research was supported by Australian Research Council
  grant DP180100595. This work was also partially supported by Simons
  Foundation Collaboration grants \#209277 (MI), \#353626 (AK) and
  \#507798 (DPW), and by Grant N0017318WR00251 from the Office of
  Naval Research and by the Naval Research Laboratory. This work was
  also facilitated by visits of the 
  second and fifth author to the University of Wollongong as well as a
visit by the second author to the Universit\'e d'Orl\'eans.  We thank
our respective hosts and their institutions for their hospitality and
support.}



\maketitle

\section*{Introduction}

There is a significant body of literature regarding the \cs-algebras
of extensions of groupoids by group bundles. The main goal of this
paper is to introduce a pushout construction for extensions of
groupoids by abelian group bundles and explore its applications.

Specifically, we consider a locally compact Hausdorff groupoid $\G$
together with an abelian group bundle $p_{\A}:\A\to\goo$ where
$p_{\A}$ a continuous, open map.  Then we consider unit space fixing
extensions
\begin{equation}
  \label{eq:formal-ext-intro}\tag{\dag}
  \begin{tikzcd}[column sep=3cm]
    \A \arrow[r,"\iota"] \arrow[dr,shift left, bend right = 15]
    \arrow[dr,shift right, bend right = 15]&\Sigma \arrow[r,"p", two
    heads] \arrow[d,shift left] \arrow[d,shift right]&\G
    \arrow[dl,shift left, bend left = 15] \arrow[dl,shift right, bend
    left = 15]
    \\
    &\goo&
  \end{tikzcd}
\end{equation}
where $\Sigma$ is a locally compact Hausdorff groupoid, both $\iota$
and $p$ are groupoid homomorphisms, $p$ is a continuous, open
surjection inducing a homeomorphism of $\go$ and $\goo$, $\iota$ is a
homeomorphism of $\A$ onto $\ker p$.

A fundamental class of such examples are $\T$-groupoids (also called
twists) introduced by the second author in \cite{kum:lnim83}.  Then
$\A$ is the trivial bundle $\goo\times\T$ such that
$\iota(r(\sigma),z)\sigma=\sigma\iota(s(\sigma),z)$ for all
$\sigma\in\Sigma$ and $z\in\T$.  These groupoids and their restricted
groupoid \cs-algebras, $\cs(\G;\Sigma)$, have enjoyed considerable
scrutiny \citelist{\cite{muhwil:ms92}, \cite{muhwil:plms395},
  \cite{kum:lnim83}, \cite{kum:cjm86}}. As usual, in this context we
often write $\dot\sigma$ in place of $p(\sigma)$.

More recently, we considered more general extensions in
\cite{iksw:jot19} and \cite{ikrsw:jfa21} as in
\eqref{eq:formal-ext-intro} where it is assumed that $\A$ is endowed
with an action of $\G$ and that the extension is compatible in the
sense
that
$\sigma\iota(a)\sigma^{-1}=\iota(\dot\sigma\cdot a)$ for all
$a \in \A$ and $\sigma \in \Sigma$ such that $p_{\A}(a) = s(\sigma)$.
% (Since $\iota(\A) = \ker p$, there is an action of $\Sigma$ on $\A$
% which descends to an action of $\G$ on $\A$, because $\A$ is abelian.)

As a consequence of the main result in \cite{ikrsw:jfa21}, we showed
that if $\Sigma$ has a Haar system, then $C^*(\Sigma)$ can be realized
as the $C^*$-algebra of a twist. Specifically, the action of $\G$ on
$\A$ induces a natural action of $\G$ on $\hA$ (regarded as a
space). We constructed a $\T$-groupoid $\Sigmaw$ of the form
\begin{equation}
  \label{eq:13a}\tag{\ddag}
  \begin{tikzcd}[column sep=3cm]
    \hA\times\T \arrow[r,"i"] \arrow[dr,shift left, bend right = 15]
    \arrow[dr,shift right, bend right = 15]&\Sigmaw \arrow[r,"j", two
    heads] \arrow[d,shift left] \arrow[d,shift
    right]&\hA\rtimes\G. \arrow[dl,shift left, bend left = 15]
    \arrow[dl,shift right, bend left = 15]
    \\
    &\hA&
  \end{tikzcd}
\end{equation}
We proved (\cite{ikrsw:jfa21}*{Theorem~3.4}) that $C^*(\Sigma)$ is
isomorphic to the restricted $C^*$-algebra
$C^*(\hA \rtimes \G; \Sigmaw)$ of this $\T$-groupoid. (In
\cite{ikrsw:jfa21} the $\T$-groupoid is denoted $\widehat{\Sigma}$,
but here we use $\Sigmaw$ to avoid possible confusion in our
examples.) The $\T$-groupoid $\Sigmaw$ is at the heart of the Mackey
obstruction which appears in the classical ``Mackey machine'' of
\cite{mac:am58}.
  
The chief motivation for this article is the observation that the
$\T$-groupoid $\Sigmaw$ above---which was based on the construction of
% was constructed in an \emph{ad hoc} fashion in
% \cite{ikrsw:jfa21}*{\S3.1}
\cite{mrw:tams96}*{Proposition~4.3}---is derived from a natural and
functorial ``pushout'' construction based on the second author's work
in \cite{kum:jot88} for \'etale groupoids (there called ``sheaf
groupoids'').  Specifically, suppose we are given and extension as in
\eqref{eq:formal-ext-intro}, and abelian group bundle $\B$ admitting a
$\G$-action, and a equivariant groupoid homomorphism $f:\A\to \B$.
Then there is a similar sort of extension
% Specifically, if we are given an extension
% \eqref{eq:formal-ext-intro} such that $\G$ acts on $\A$ as above and a
% groupoid homomorphism $f:\A\to\B$, where $\B$ is an abelian group
% bundle admitting a $\G$-action such that $f$ is equivariant, there is
% a similar sort of extension
\begin{equation}
  \label{eq:-ext-fstar}
  \begin{tikzcd}[column sep=3cm]
    \B \arrow[r,"\iota"] \arrow[dr,shift left, bend right = 15]
    \arrow[dr,shift right, bend right = 15]&f_{*}\Sigma \arrow[r,"p",
    two heads] \arrow[d,shift left] \arrow[d,shift right]&\G
    \arrow[dl,shift left, bend left = 15] \arrow[dl,shift right, bend
    left = 15]
    \\
    &\goo&
  \end{tikzcd}
\end{equation}
inducing the given $\G$-action on $\B$.  In Theorem~\ref{thm:functor},
we show that the construction producing $f_{*}\Sigma$ has good
functorial properties that characterize the extension up to a suitable
notion of isomorphism.  Using these properties, we show in
Theorem~\ref{thm:functoriali} that the collection $\Tga$ of
isomorphism classes of extensions of $\A$ by $\G$ forms an abelian
group (see also \cite{tu:tams06}*{\S{5.3}}).

We close by illustrating how the pushout construction clarifies and
interacts with our work in \cite{iksw:jot19} and \cite{ikrsw:jfa21}.
In Theorem~\ref{thm:pushout} we prove that the
extension~\eqref{eq:13a} employed in \cite{ikrsw:jfa21} arises from
our pushout construction. Specifically, the natural pairing
$(\chi, a) \mapsto \chi(a)$ from $\hA * \A$ to $\T$ yields a groupoid homomorphism
$f: \hA * \A \to \hA\times\T$ given by $f(\chi, a) = (\chi, \chi(a))$ 
(see Section~\ref{sec:t-groupoid}).
There is a natural action of $\Sigma$ on $\hA$ (compatible with that
of $\G$ as above) and we prove that
$\Sigmaw \cong f_*(\hA \rtimes \Sigma)$.  This allows us to realise
the $C^*$-algebra of an extension of a groupoid $\G$ by an abelian
group bundle $\A$ as the $C^*$-algebra of a $\T$-groupoid over the
resulting transformation groupoid $\hA \rtimes \G$.
%% for the natural action of $\G$ on the Gelfand dual $\hA$ of
%% $C^*(\A)$.

Several consequences flow from this observation.  % If $\Sigma$ is a
% generalized twist, that is, $\A = \goo \times A$ where $A$ is a
% locally compact abelian group and the action of $\G$ on $\A$ is
% trivial on the second factor, we prove that $\cs(\Sigma)$ is the
% section algebra of an upper-semicontiuous $\cs$-bundle over $\hat{A}$
% with fibers of the form $\cs(\G; \fchiss)$ where $\fchiss$ is the
% pushout of $\Sigma$ determined by $\chi \in \hat{A}$ (see
% Proposition~\ref{prop:bundle}).
First suppose that A is an abelian group and that
$\A = \goo \times A$, carrying the action of $\G$ that is trivial
in the second coordinate, so that $\Sigma$ is a generalised twist. Each
$\chi \in \hat{A}$ defines a homomorphism
$\fchi: \A \to \T \times \G^{(0)}$, so we can form the resulting
pushout $\fchiss$. We prove in Proposition~\ref{prop:bundle} that
$C^*(\Sigma)$ is the section algebra of an upper-semicontinuous
$C^*$-bundle over $\hat{A}$ with fibres $C^*(\G, \fchiss)$.
When $A$ is compact, this yields a direct sum decomposition which
remains valid for the corresponding reduced $\cs$-algebras (see
Proposition~\ref{prop:compact-case}).  In
Corollary~\ref{cor:twisted-extensions} we extend
\cite{ikrsw:jfa21}*{Theorem~3.4} to the case that $\Omega$ is a
$\T$-groupoid extension of $\Sigma$ such that its restriction to
$\iota(\A)$ is abelian.  When $\G$ is {\'e}tale, this enables us
to generalize \cite[Theorem 4.6]{ikrsw:jfa21} to this case (see
Corollary~\ref{cor:twisted-cartan}) thereby providing criteria that
guarantee that the natural abelian subalgebra of
$\cs_r(\Sigma; \Omega)$ is Cartan (see also \cite[Theorem
5.8]{dgnrw:jfa20} and \cite[Theorem 4.6]{dgn:xx20}).

In Subsection~\ref{sec:twists-cocycles}, we consider the case where
the extension $\Sigma$ is determined by an $\A$-valued 2-cocycle
defined on $\G$ and show that the pushout construction is compatible
with the natural change of coefficients map on
cocycles.  % The explicit construction of $\Sigmaw$
% in terms of 2-cocycles is given at the beginning of subsection 3.3 and
% various examples of this construction are considered.
We describe the explicit construction of $\Sigmaw$ in
terms of $2$-cocycles at the beginning of
Subsection~\ref{sec:t-groupoid-2}, and then 
consider various examples of this construction. 
 
\section{Pushouts of groupoid extensions}
\label{sec:push}

We fix a locally compact Hausdorff groupoid $\G$. In our applications,
$\G$ will have a Haar system, but this is not required for the pushout
construction itself.  However, we do assume that $\G$ has open range
and source maps.  We call a locally compact abelian group bundle
$p_{\A}:\A\to\goo$ a \emph{$\G$-bundle} if $p_{\A}$ is open and $\G$
acts on the left of $\A$ by automorphisms.  For compatibility with
\cite{iksw:jot19}---and other examples we have in mind---we will write
the group operations in the fibres of such $\A$ additively.  An
extension $\Sigma$ of $\A$ by $\G$ is determined by a diagram
\eqref{eq:formal-ext-intro} as in the introduction.
% \begin{equation} %% No reason to repeat.
%   \label{eq:formal-ext}
%   \begin{tikzcd}[column sep=3cm]
%     \A \arrow[r,"\iota"] \arrow[dr,shift left, bend right = 15]
%     \arrow[dr,shift right, bend right = 15]&\Sigma \arrow[r,"p", two
%     heads] \arrow[d,shift left] \arrow[d,shift right]&\G
%     \arrow[dl,shift left, bend left = 15] \arrow[dl,shift right,
%     bend left = 15]
%     \\
%     &\goo&
%   \end{tikzcd}
% \end{equation}
Recall that $\Sigma$ is a locally compact Hausdorff groupoid, $p$ is
continuous and open surjection inducing a homeomorphism from $\go$ onto
$\goo$, and $\iota$ is a continuous open injective homomorphism onto
$\ker p=\set{\sigma\in\Sigma:p(\sigma)\in\goo}$.  We call such an
extension \emph{compatible} if the action of $\G$ on $\A$ induced by
conjugation is the given $\G$-action on $\A$; that is,
$\sigma\iota(a)\sigma^{-1} =\iota(\dot\sigma\cdot a)$.% where as usual
% we have written $\dot\sigma$ in place of $p(\sigma)$.

\begin{definition}
  \label{def-prop-iso} If $\Sigma_{1}$ and $\Sigma_{2}$ are compatible
  extensions by a locally compact abelian group $\G$-bundle $\A$, then
  we say that they are \emph{properly isomorphic} if there is a
  groupoid isomorphism $f:\Sigma_{1}\to\Sigma_{2}$ such that the
  diagram
  \begin{equation}
    \label{eq:73}
    \begin{tikzcd}[row sep = 1ex, column sep = 2cm]
      &\Sigma_{1} \arrow[dd,"f"] \arrow[rd,"p_{1}"] \\
      \A \arrow[ru,"\iota_{1}"] \arrow[rd,"\iota_{2}",swap] && \G \\
      &\Sigma_{2}\arrow[ru,"p_{2}",swap] \\
    \end{tikzcd}
  \end{equation}
  commutes.  We let $\Tga$ be the collection of proper isomorphism
  classes of compatible extensions. We denote the equivalence class of
  a compatible extension $\Sigma$ by $[\Sigma]$.
\end{definition}

\begin{remark}\label{rmk:Kum-T-extensions}
  The second author considered extensions of this sort for \'etale
  groupoids in \cite{kum:jot88}*{\S{2}}. In
  \cite{tu:tams06}*{\S{5.3}}, Tu denotes this set by $\ext(\G, \A)$
  and states that it forms an abelian group (see Theorem
  \ref{thm:functoriali} below).  Since the openness of
  $p_{\A}:\A\to\goo$ implies that $\A$ has a Haar system (see
  \cite{ikrsw:jfa21}*{Lemma~2.1}), it follows that if $\G$ has a Haar
  system, then we can then equip $\Sigma$ with a Haar system whenever
  $[\Sigma]\in\Tga$ (see \cite{ikrsw:jfa21}*{Lemma~2.6}).
\end{remark}

Of course, given $\G$ and a $\G$-bundle $\A$, we would like to know
that $\Tga$ is not empty.  To provide a basic example, we follow
\cite{kum:jot88}*{Definition~2.1}.

\begin{example}
  [The Semidirect Product] \label{ex-semi-direct} We can build a
  fundamental compatible extension $\SD(\A,\G)$ from the fibred
  product $\set{(a,\gamma)\in\A\times \G:p_{\A}(a)=r(\gamma)}$.  We
  let
  $(\SD(\A,\G))^{(2)}=
  \bigl\{\,\bigl((a_{1},\gamma_{1}),(a_{2},\gamma_{2})\bigr):s(\gamma_{1})
  =r(\gamma_{2})\,\bigr\}$, and then define
  \begin{align}
    \label{eq:60}
    (a_{1},\gamma_{1})(a_{2},\gamma_{2})= (a_{1}+\gamma_{1}\cdot
    a_{2},\gamma_{1}\gamma_{2})\quad\text{and}\quad (a,\gamma)^{-1}=
    (-(\gamma^{-1}\cdot a),\gamma^{-1}).
  \end{align}
  Then we can identify the unit space of $\SD(\A,\G)$ with $\goo$ so
  that $r(a,\gamma)=r(\gamma)$ and $s(a,\gamma)=s(\gamma)$.  We can
  then exhibit $\SD(\A,\G)$ as an extension by letting
  $\iota(a)=(a,p_{\A}(a))$, and letting $p(a,\gamma)=\gamma$. Since
  \begin{equation}
    \label{eq:72}
    (a',\gamma)(a,p_{\A}(a))(-\gamma^{-1}\cdot a',\gamma^{-1}) =
    (\gamma\cdot a,p_{\A}(\gamma\cdot a)),
  \end{equation}
  $\SD(\A,\G)$ is a compatible extension as required.
\end{example}

\begin{example}\label{ex-fibred}
  For $i = 1, 2$ let $\A_i$ be a locally compact abelian group
  $\G$-bundle.  Note that
  $\A_{1}*\A_{2}=\set{(a,a'):p_{\A_{1}}(a)=p_{\A_{2}}(a')}$ is also a
  locally compact abelian group $\G$-bundle.  Let $\Sigma_i$ be a
  compatible groupoid extension of $\G$ by $\A_i$.  Then as in
  \cite{kum:jot88}*{\S{2}}, we may form the fibered product
  \[
    \Sigma_1 *_\G \Sigma_2 := \{ (\sigma_1, \sigma_2) \in \Sigma_1
    \times \Sigma_2 \mid p_1(\sigma_1) = p_2(\sigma_2) \}.
  \]
  It is straightforward to check that $\Sigma_1 *_\G \Sigma_2$ is a
  compatible groupoid extension of $\G$ by $\A_1 * \A_2$.
\end{example}

Assume now that $\B$ is another abelian group $\G$-bundle, and that
$f:\A\to\B$ is a %open
$\G$-equivariant map. Following \cite{kum:jot88}*{Proposition~2.6}, we
prove that we can ``pushout'' $\Sigma$ in a unique way to an extension
of $\G$ by $\B$.

\begin{thm}[Pushout Construction]\label{thm:functor}
  Let $\A$ and $\B$ be locally compact abelian group $\G$-bundles. Let
  $f:\A\to \B$ be a continuous %open
  $\G$-equivariant map. Assume that $\Sigma$ is a compatible extension
  of $\G$ by $\A$. Then there is a compatible extension $f_*\Sigma$ of
  $\G$ by $\B$ and a homomorphism $f_*:\Sigma\to f_*\Sigma$ such that
  the following diagram commutes
  \begin{equation}
    \label{eq:74}
    \begin{tikzcd}[row sep = 1ex, column sep = 2cm]
      \A\arrow[r,"\iota"] \arrow[dd,"f"] &\Sigma \arrow[dd,"f_{*}"]
      \arrow[rd,"p"] \\
      && \G .\\
      \B \arrow[r,"\iota_{*}",swap]
      &f_{*}\Sigma\arrow[ru,"p_{*}",swap] \\
    \end{tikzcd}
  \end{equation}
  Moreover, $f_{*}$ and $f_{*}\Sigma$ are unique up to proper
  isomorphism in the sense that if $\Sigma'$ is a compatible extension
  such that the diagram
  \begin{equation}
    \label{eq:84}
    \begin{tikzcd}[row sep = 1ex, column sep = 2cm]
      \A\arrow[r,"\iota"] \arrow[dd,"f"]& \Sigma \arrow[dd,"f'"]
      \arrow[rd,"p"] \\
      && \G \\
      \B\arrow[r,"\iota'"] & \Sigma' \arrow[ru,"p'"]
    \end{tikzcd}
  \end{equation}
  commutes, then there is a proper isomorphism
  $g:f_{*}\Sigma\to \Sigma'$ such that $g\circ f_{*}=f'$.
\end{thm}
\begin{proof}
  Consider the fibred-product groupoid
  \[
    \D:= (\SD(\B, \G))*_\G \Sigma=\{\,(( b,\gamma),\sigma )\in
    (\SD(\B,\G))\times \Sigma\, : \text{$\dot\sigma=\gamma$}\,\}
  \]
  of Example~\ref{ex-fibred}.  Define $\theta:\A\to \D$ via
  $\theta(a)=\bigl( (-f(a),p_{\A}(a)),\iota(a)\bigr)$. Since $\iota$
  is a homeomorphism onto its closed range, $\theta(\A)$ is a closed
  wide subgroupoid of $\D$.

  Let $d=((b,\gamma),\sigma)\in \D$.  We claim that
  $d\theta(\A)=\theta(\A)d$.  To see this, note that
  \begin{align}
    \label{eq:3d}
    d\theta(a)
    &=
      \bigl((b,\gamma),\sigma)((-f(a),p_{\A}(a)),\iota(a)\bigr)
    \\
    &= \bigl((b-\gamma\cdot f(a),\gamma),\sigma\iota(a)
      \bigr)  \\
    &= \bigl((-f(\gamma\cdot a)+p_{\A}(\gamma\cdot a)\cdot
      b,\gamma),\iota(\dot\sigma\cdot
      a)\sigma\bigr). \\
    \intertext{Since $\dot\sigma=\gamma$, we deduce that}
    d\theta(a)
    &= \bigl((-f(\gamma\cdot a),p_{\A}(\gamma\cdot
      a)),\iota(\gamma\cdot a)\bigr)
      (b,\gamma,\sigma) \\
    &= \theta(\gamma\cdot a) d.
  \end{align}
  Let $f_*\Sigma:=\D/\theta(\A)$.  As usual, we denote the class of
  $((b,\sigma),\gamma)$ in $f_{*}\Sigma$ by $[(b,\sigma),\gamma]$.
  Then $[(b,\gamma),\iota(a)\sigma] = [(b+ f(a),\gamma),\sigma]$.
  Since $j(\A)$ has a Haar system by
  Remark~\ref{rmk:Kum-T-extensions}, $f_{*}\Sigma$ is a locally
  compact Hausdorff groupoid by \cite{ikrsw:jfa21}*{Lemma~2.2}.  The
  operations are given by
  \begin{align}
    [(b_1,\gamma_1),\sigma_1][(b_2,\gamma_2),\sigma_2]
    &=
      [(b_1+\gamma_1b_2,\gamma_1\gamma_2),\sigma_1\sigma_2]  \quad
      \text{and} \\ \relax [(b,\gamma),\sigma]^{-1}
    &= [(-\gamma^{-1}\cdot
      b,\gamma^{-1}),\sigma^{-1}].
  \end{align}
  We can identify the unit space with $\goo$ and then
  \begin{align}
    \label{eq:75}
    r([ (b,\gamma),\sigma])=r(\gamma) \quad\text{and}\quad s([
    (b,\gamma),\sigma])=s(\gamma) .
  \end{align}

  To see that $f_*\Sigma$ is a compatible extension by $\B$, let
  \[
    \iota_*(b)=[(b,p_{\B}(b)),p_{\B}(b)]\quad\text{ and }\quad
    p_*([(b,\gamma),\sigma])=\gamma.
  \]
  It is not hard to verify that this satisfies the algebraic
  requirements for an extension. The most difficult one might be the
  inclusion $p_*^{-1}(\goo)\subseteq \iota_*(\B)$ for which we provide
  an outline of the proof: take $[ (b,\gamma ),\sigma]\in f_*\Sigma$
  such that $p_*([ (b,\gamma),\sigma])=u\in\goo$. Then $\gamma=u$,
  giving $\dot\sigma=u$. Since $\Sigma$ is an extension, there exists
  $a\in \A(u)$ such that $\iota(a)=\sigma$. It follows that
  $[( (b,u),\iota(a))]=[( (b+f(a),u),u)]=\iota_*(b+f(a))$. It is easy
  to check that $b+f(a)$ is independent of the choice of the
  representative of $[ (b,\gamma),\sigma]$.

  Since $\iota_{*}$ and $p_{*}$ are clearly continuous and since
  $\iota_{*}$ is easily seen to be a homeomorphism onto its range, we
  just need to see that $p_{*}$ is open.  For this, we apply Fell's
  Criterion (see \cite{ikrsw:jfa21}*{Lemma~3.1}).  Suppose that
  $\gamma_{n}\to \gamma= p_{*}\bigl([(b,\sigma),\gamma]\bigr)$.  Since
  $p:\Sigma\to\G$ is open, we can pass to a subnet, relabel, and
  assume that there are $\sigma_{n}\to \sigma$ in $\Sigma$ such that
  $\dot\sigma_{n}=\gamma_{n}$.  Since $p_{\B}$ is open, we can pass to
  subnet, relabel, and assume there are $b_{n}\to b$ in $\B$ such that
  $p_{\B}(b_{n}) =r(\gamma_{n})$.  Then
  $[(b_{n},\gamma_{n}),\sigma_{n}]\to [(b,\gamma),\sigma]$ as
  required.

  The map $f_*$ is the composition of the embedding of $\Sigma$ into
  $\D$ and the quotient map $\D\mapsto \D/i(\A)$:
  $f_*(\sigma)=[\bigl( ( 0_{r(\sigma)}, p(\sigma) ),\sigma\bigr)]$.
  Since $f$ is $\G$-equivariant, $p_{\B}(f(a))= p_{\A}(a)$ and
  \begin{align}
    \label{eq:76}
    f_{*}(\iota(a))=[(0,p_{\A}(a)),p_{\A}(a)]=
    [(f(a),p_{\B}(f(a))),p_{\B}(f(a))]
    = \iota_{*}(\iota(a)),
  \end{align}
  and \eqref{eq:74} commutes as required.

  Now let $\Sigma'$ be an extension as in~\eqref{eq:84}.  Define
  $\tilde g:\D\to\Sigma'$ by
  $\tilde g((b,\gamma),\sigma)= \iota'(b)f'(\sigma)$.  Since
  \begin{equation}
    \label{eq:6b}
    \iota'(b_{1})f'(\sigma_{1})\iota'(b_{2})f'(\sigma_{2}) =
    \iota'(b_{1}) \iota'(f'(\sigma_{1})\cdot b_{2})
    f'(\sigma_{1})f'(\sigma_{2})
  \end{equation}
  and since $p'(f'(\sigma_{1}))=\dot\sigma_{1}$, it follows that
  $\tilde g$ is a groupoid homomorphism.  On the other hand,
  \begin{align}
    \label{eq:7b}
    \tilde g(\theta(a)) &= \tilde g((-f(a),p_{\A}(a)),\iota(a))
                          = \iota'(-f(a))f'(\iota(a))
                          = \iota'(-f(a))\iota'(f(a))
    \\
                        &= \iota'(p_{\A}(a)).
  \end{align}
  Hence $\tilde g$ factors through a homomorphism
  $g:f_{*}\Sigma\to\Sigma'$.  Clearly, $g(\iota_{*}(b))=\iota'(b)$ and
  $p'\circ g=p_{*}$, so $g$ makes the diagram analogous
  to~\eqref{eq:73} commute. We have $g\circ f_{*}=f'$ by construction.

  To see that $g$ is a proper isomorphism, we still need to see that
  $g$ is an isomorphism with a continuous inverse.

  For this, fix $\alpha\in\Sigma'$.  There exists $\sigma\in\Sigma$
  such that $p(\sigma)=p'(\alpha)$.  Using \eqref{eq:84}, there exists
  $b\in\B$ such that $\alpha=\iota'(b)f'(\sigma)$. So $\tilde{g}$, and
  hence also $g$, is onto.

  Now suppose that $\iota'(b)f'(\sigma)$ is a unit.  Then
  $f'(\sigma)= \iota'(-b)$.  Hence $p'(f'(\sigma))$ is a unit, and
  $\sigma=\iota(a)$ for some $a\in\A$.  But then
  $\iota'(-b)=f'(\sigma)=f'(\iota(a))=\iota'(f(a))$.  Hence,
  $b=-f(a)$.  That is,
  \begin{equation}
    \label{eq:8b}
    ((b,p(\sigma)),\sigma)=((-f(a),p_{\A}(a)),\iota(a))\in\theta(\A).
  \end{equation}
  Thus $g$ is injective.

  To see that $g$ is an isomorphism of topological groupoids, it
  suffices to see that $g$ is open.  We use Fell's criterion.  So
  suppose that $g(\alpha_{i})\to g(\alpha)$ where
  $\alpha_i=[(b_i,p(\sigma_i)),\sigma_i]$ and
  $\alpha=[(b,p(\sigma)),\sigma]\in f_*\Sigma$.  Since
  $p'\circ g=p_{*}$, we have $p(\sigma_{i})\to p(\sigma)$.  Since $p$
  is open, we can pass to a subnet, relabel, and assume there exist
  $a_{i}\in\A$ such that $\iota(a_{i})\sigma_{i} \to \sigma$.  But
  \begin{equation}
    \label{eq:9b}
    \alpha_{i}=[(-f(a_{i})+b_{i}),p(\sigma_{i}),\iota(a_{i})\sigma_{i}],
  \end{equation}
  and then
  \begin{equation}
    \label{eq:10b}
    \iota'(-f(a_{i})+b_{i})f'(\iota(a_{i})\sigma_{i})\to
    \iota'(b)f'(\sigma).
  \end{equation}
  It follows that
  \begin{equation}
    \label{eq:11c}
    \iota'(-f(a_{i})+b_{i}) \to \iota'(b).
  \end{equation}
  Since $\iota'$ is a homeomorphism onto its range,
  $\alpha_{i}\to\alpha$ as required.
\end{proof}

\begin{cor}\label{cor:compositions}
  Let $\A$, $\B$ and $\mathcal C$ be locally compact abelian group
  $\G$-bundles. Let $f:\A\to \B$ and $g: \B \to \mathcal C$ be
  continuous $\G$-equivariant maps.  Assume that $\Sigma$ is a
  compatible extension of $\G$ by $\A$.  Then $(g \circ f)_*\Sigma$ is
  properly isomorphic to $g_*(f_*\Sigma)$.
\end{cor}
\begin{proof}
  This follows from the uniqueness of $(g \circ f)_*\Sigma$ up to
  proper isomorphism guaranteed by Theorem~\ref{thm:functor}.
\end{proof}.

\section{The Extension Group $\Tga$}
\label{sec:extension-group-tga}

As in \cite{kum:jot88}*{\S2}, we can use our pushout construction to
introduce a binary operation on $\Tga$.  Suppose that
$[\Sigma],[\Sigma']\in \Tga$.  Define $\nabla^{\A}:\A*\A\to\A$
by $\nabla^{\A}(a,a')=a+a'$. Proper isomorphisms $f : \Sigma \to \Gamma$
and $f' : \Sigma' \to \Gamma'$ of compatible extensions of $\A$ by $\G$ determine a
proper isomorphism $f * f' : \Sigma*\Sigma' \to \Gamma*\Gamma'$ of extensions by $\A * \A$. 
The uniqueness assertion of Theorem~\ref{thm:functor} then yields a
proper isomorphism $\nabla^{\A}_{*}(\Sigma*_{\G}\Sigma') \to
\nabla^{\A}_{*}(\Gamma*_{\G}\Gamma')$. Hence the formula 
\begin{align}
  \label{eq:85}
  [\Sigma]+[\Sigma'] :=
  [\nabla^{\A}_{*}(\Sigma*_{\G}\Sigma')]
\end{align}
is well defined.

\begin{example}\label{ex-iden}
  Let $[\Sigma]\in \Tga$. Let
  $\SD(\A, \G)$ be the semidirect product defined in
  Example~\ref{ex-semi-direct}. Define $g : (\SD(\A,\G))*_{\G}\Sigma 
  \to \Sigma$ by $g((a,\dot\sigma),\sigma)=\iota(a)\sigma$. We obtain 
  a commutative diagram
  \begin{equation}
    \label{eq:78}
    \begin{tikzcd}[row sep = 1ex, column sep = 2cm]
      \A*\A\arrow[r,"\iota*\iota"] \arrow[dd,"\nabla^{\A}",swap]&
      (\SD(\A,\G))*_{\G}\Sigma
      \arrow[dr,"\tilde p"] \arrow[dd,"g",swap]\\
      &&\G. \\
      \A\arrow[r,"\iota",swap]&\Sigma \arrow[ur,"p",swap]
    \end{tikzcd}
  \end{equation}
  The uniqueness assertion in Theorem~\ref{thm:functor} implies that
  $\nabla^{\A}_{*}((\SD(\A,\G))*_{\G}\Sigma)$ is properly isomorphic to
  $\Sigma$.  In other words, $[\SD(\A,\G)]+[\Sigma]=[\Sigma]$.
\end{example}

\begin{example}
  \label{ex-theta}
  Let $\A\overset{\iota}\longrightarrow \Sigma
  \overset{p}\longrightarrow \G$ be a compatible extension.  Then we
  obtain another 
  compatible extension $\A\overset{\iota'}\longrightarrow \Sigma
  \overset{p}\longrightarrow \G$ by letting
  $\iota'(a)=\iota(-a)=\iota(a)^{-1}$.  We 
  will write $\Sigma^{-1}$ for $\Sigma$ viewed as this alternate
  extension.  Define $\theta:\A\to\A$ by $\theta(a)=-a$.  
  Then $\theta$ is $\G$-invariant.  Since the diagram
  \begin{equation}
    \label{eq:81}
    \begin{tikzcd}
      [row sep = 1ex, column sep = 2cm]
      \A\arrow[r,"\iota"]\arrow[dd,"\theta"] &\Sigma \arrow[dd,"\id"]
      \arrow[rd,"p"] \\
      &&\G \\
      \A \arrow[r,"\iota'"] &\Sigma^{-1}\arrow[ru,"p"] \\
    \end{tikzcd}
  \end{equation}
  commutes, we can identify $[\theta_{*}\Sigma]$ with $[\Sigma^{-1}]$
  by Theorem~\ref{thm:functor}.
\end{example}

\begin{example}\label{ex-inv}
  Take $[\Sigma] \in \Tga$ and let
  $\SD(\A, \G)$ be the semidirect product. The map
  $g:\Sigma*\Sigma^{-1}\to \SD(\A,\G)$ given by
  $g(\sigma,\tau)= (\iota^{-1}(\sigma\tau^{-1}),\dot\sigma)$ is a
  homomorphism. Since the diagram
  \begin{equation}
    \label{eq:79}
    \begin{tikzcd}
      [row sep = 1ex, column sep = 2cm] \A*\A\arrow[r,"\iota*\iota'"]
      \arrow[dd,"\nabla^{\A}",swap]& \Sigma*_{\G}\Sigma^{-1}
      \arrow[dr,"\tilde p"] \arrow[dd,"g",swap]\\
      &&\G \\
      \A\arrow[r,"\iota",swap]&\SD(\A,\G) \arrow[ur,"p",swap]
    \end{tikzcd}
  \end{equation}
  commutes, we see that $[\Sigma]+[\Sigma^{-1}]=[\SD(\A,\G)]$ for all
  $\Sigma\in \Tga$.
\end{example}

\begin{example}\label{ex-comm}
  Take $[\Sigma],[\Sigma']\in\Tga$.  Let
  $\tilde f:\Sigma*_{\G}\Sigma' \to \Sigma'*_{\G}\Sigma$ be the flip.
  Similarly, let $f:\A*\A\to\A*\A$ be given by $f(a,a')=(a',a)$.  The
  diagram
  \begin{equation}
    \label{eq:86}
    \begin{tikzcd}[row sep = 4ex, column sep = 2cm]
      \A*\A \arrow[r,"\iota*\iota'"] \arrow[d,"f",swap] &
      \Sigma*_{\G}\Sigma' \arrow[d,"\tilde f",swap] \arrow[rd,"\tilde
      p"]
      \\
      \A*\A \arrow[r,"\iota'*\iota"] \arrow[d,"\nabla^{\A}",swap] &
      \Sigma'*_{\G} \Sigma \arrow[r,,"\tilde p"]
      \arrow[d,"\nabla^{\A}_{*}",swap]
      & \G \\
      \A \arrow[r,"i"] & \nabla_{*}(\Sigma'*_{\G}\Sigma)
      \arrow[ru,"p",swap]
    \end{tikzcd}
  \end{equation}
  commutes.  Since $\nabla^{\A} \circ f = \nabla^{\A}$, it follows
  from Theorem~\ref{thm:functor} that
  $[\Sigma]+[\Sigma']=[\Sigma']+[\Sigma]$.
\end{example}

In Examples~\ref{ex-iden}--\ref{ex-comm}, we have proved much of the following theorem, which is patterned on 
\cite{kum:jot88}*{Theorem 2.7}.

\begin{thm}\label{thm:functoriali}
  Let $\G$ be a locally compact groupoid with open range and source
  maps, and let
  $\A$ be a locally compact abelian group $\G$-bundle.  Then the
  binary operation
  $([\Sigma_1], [\Sigma_2]) \mapsto [\nabla^{\A}_{*}(\Sigma_1 *_\G
  \Sigma_2)]$ of~\eqref{eq:85} makes $\Tga$ into an abelian group with
  neutral element given by the class $[\SD(\A, \G)]$ of the
  semidirect product of Example~\ref{ex-semi-direct}, and
  $[\Sigma]^{-1}=[\Sigma^{-1}]$ as in Example~\ref{ex-theta}.   
  For each continuous  $\G$-equivariant map $f:\A\to\B$ of $\G$-bundles,
  define $\Tg(f):\Tg(\A)\to\Tg(\B)$ to be the induced map: 
  $\Tg(f)[\Sigma]=[f_{*}\Sigma]$. Then $\Tg$ is a functor from the category 
  of $\G$-bundles to the category of abelian groups. 
\end{thm}

\begin{proof}
  By considering diagrams similar to that in Example~\ref{ex-comm}, we
  see that the operation in \eqref{eq:85} is well-defined and
  associative.  We saw that $[\SD(\A,\G)]$ acts as an identity in
  Example~\ref{ex-iden} and the statement about inverses follows from
  Example~\ref{ex-inv}.  The computation in Example~\ref{ex-comm} 
  shows that $\Tga$ is an abelian group.
   
  By Corollary~\ref{cor:compositions} we have
  $\Tg(f \circ g) = \Tg(f) \circ \Tg(g)$ if $f$ and $g$ are a
  composable pair of continuous $\G$-equivariant maps of $\G$-bundles.
  The proof that $\Tg(f)$ is a group homomorphism follows as in the proof of
  \cite{kum:jot88}*{Theorem 2.7}.
\end{proof}

\section{Applications and Examples}
\label{sec:appl-cocycl}

In this section we consider a unit space fixing extension $\Sigma$
of $\G$ by the group bundle $\A$ as illustrated in  the
diagram~(\ref{eq:formal-ext-intro}) from the introduction. We review
the basic details.  We assume that all groupoids considered in this
section are second-countable locally compact Hausdorff groupoids with
Haar systems.  The Haar system on $\Sigma$ is denoted
$\lambda=\sset{\lambda^{u}}_{u\in\go}$ and we further assume that
$p_{\A}:\A\to\go$ is a bundle of abelian groups that is a closed
subgroupoid of $\Sigma$.  It is equipped with a Haar system denoted
$\beta=\sset{\beta^{u}}_{u\in\go}$ and the fibers are denoted
$\A(u)$ for $u\in\go$.  The existence of a Haar system on $\A$
implies that $p_{\A}$ is open.  It follows by
\cite{ikrsw:jfa21}*{Lemma 2.6(c)} that there is a Haar system
$\alpha=\sset{\alpha_{u}}_{u\in\go}$ on $\G$ such that for all
$f\in\cc(\Sigma)$ and $u\in\go$ we have
\begin{equation}
  \label{eq:harhar}
  \int_{\Sigma}f(\sigma)\,d\lambda^{u}(\sigma) = \int_{\G}
  \int_{\A} f(\sigma a)\,d\beta^{s(\sigma)}(a)\, d\alpha^{u}(\dot \sigma).
\end{equation}
Moreover, there is a natural action of $\Sigma$, and therefore $\G$,
on $\A$.

Note that $p:\Sigma\to \G$ is a continuous, open surjection inducing a
homeomorphism from $\go$ onto $\goo$, and $\iota:\A\to\Sigma$ is a
homeomorphism onto $\ker p$.  (Both $p$ and $\iota$ are assumed to be
groupoid morphisms).

Recall that if $\Sigma$ is a $\T$-groupoid over $\G$ then
\[
C_c(\G; \Sigma)  := \{ f \in C_c(\Sigma) : f(t\sigma) = tf(\sigma) \text{ for all } t \in \T, \sigma \in \Sigma \}
\]
is a $^*$-algebra under the operations described in \cite{muhwil:ms92}*{\S2}, and $\cs(\G; \Sigma)$ is its closure
in the norm obtained by taking the supremum of the operator norm under all $*$-representations. 

We may also view $C_c(\G; \Sigma)$ as compactly supported continuous sections of the one-dimensional Fell line bundle over $\G$ associated to $\Sigma$.  
One can then construct the associated (right) Hilbert $C_0(\goo)$-module (see \cite{ikrsw:jfa21}*{\S 1.3}) as the completion of 
$C_c(\G; \Sigma)$ in the norm arising from the $C_0(\goo)$-valued pre-inner product given by 
$\langle f, g \rangle := (f^**g)|_{\G^{(0)}}$ for all $f, g \in C_c(\G; \Sigma)$.  
We denote the Hilbert module by $\H(\G; \Sigma)$ and observe that left multiplication induces a natural $*$-homomorphism 
$\lambda: C_c(\G; \Sigma) \to \L(\H(\G; \Sigma))$.  
We may define the reduced norm of an element $f \in C_c(\G; \Sigma)$  to be the operator norm of its image: 
$\|f\|_r := \| \lambda(f) \|$.   
Then $\cs_r(\G; \Sigma)$ is the closure of $C_c(\G; \Sigma)$ in the reduced norm.

\begin{lemma}\label{lem:inv-clopen}
With notation as above, let $F \subset \goo$ be a $\G$-invariant clopen subset.  
Then $F$ is also $\Sigma$-invariant and the reduction $\Sigma\restr{F}$ is a twist over the reduction $\G\restr F$.
Moreover, the characteristic function of $F$ determines a central multiplier projection $p_F$ such that
   \[
   p_F\cs_r(\G; \Sigma) \cong  \cs_r(\G\restr F ; \Sigma\restr F).
   \] 
\end{lemma}
\begin{proof}
Observe that $\H(\G; \Sigma)$ decomposes as the direct sum of a Hilbert $C_0(F)$-module and a 
Hilbert $C_0(F^c)$-module in the following way
\[
\H(\G; \Sigma) \cong \H(\G\restr F; \Sigma\restr F) \oplus  \H(\G\restr {F^c}; \Sigma\restr {F^c}).
\]
Note that multiplication by the characteristic function of $F$, which we denote by $p_F$  is the projection onto the first component,
 that  $p_F$ is in the center of the multiplier algebra of $\cs_r(\G; \Sigma)$, and 
$C_c(\G\restr F; \Sigma\restr F)$ acts trivially on the second component.
Hence the operator norm of $C_c(\G\restr F; \Sigma\restr F)$ acting on $\H(\G\restr F; \Sigma\restr F)$
coincides with that of its action on $\H(\G; \Sigma)$.
\end{proof}
\subsection{The $\T$-groupoid of an extension}
\label{sec:t-groupoid}

As noted in the introduction, we want to see that
the $\T$-groupoid constructed in
\cite{ikrsw:jfa21}*{\S3.1} is an example of the pushout construction
of Theorem~\ref{thm:functor}.  The \cs-algebra
$\cs(\A)$ is abelian and the Gelfand dual of $\cs(\A)$ is an abelian group
bundle $\hp:\hA\to\goo=\go$ with fibres
$\hp^{-1}(\sset u)\cong \A(u)^{\wedge}$
(see \cite{mrw:tams96}*{Corollary~3.4}).  Furthermore, since abelian groups
are amenable, it follows from \cite{wil:toolkit}*{Corollary~5.39} and
\cite{wil:crossed}*{Proposition~C.10} that $\hp$ is open.  Therefore
we can view $\hA$ as a right $\G$-bundle for the natural right action of
$\G$ on $\hA$.  

Since $\G$ and $\Sigma$ both act on $\hA$,
regarded as a topological space fibered over $\go$, we can form the
transformation groupoids $\hA\rtimes \G$ and $\hA\rtimes \Sigma$. 
Moreover, $\hA*\A= \set{(\chi,a):\hp(\chi)=p_{\A}(a)}$ is  a 
$\hA\rtimes \G$-bundle (as well as an $\hA\rtimes \Sigma$-bundle). Defining 
$\iota_* : \hA * \A \to \hA \rtimes \Sigma$ by $\iota_*(\chi,a)=(\chi,a)$ and 
$p_* : \hA \rtimes \Sigma \to \hA \rtimes \G$ by 
$p_*(\chi,\sigma)=(\chi,\dot\sigma)$, we obtain an extension
\begin{equation}\label{eq:ext1}
  \begin{tikzcd}[column sep=3cm]
    \hA* \A \arrow[r,"\iota_*", hook] \arrow[dr,shift left, bend right
    = 15] \arrow[dr,shift right, bend right = 15]&\hA\rtimes\Sigma
    \arrow[r,"p_*", two heads] \arrow[d,shift left] \arrow[d,shift
    right]&\hA\rtimes\G. \arrow[dl,shift left, bend left = 15]
    \arrow[dl,shift right, bend left = 15]
    \\
    &\hA&
  \end{tikzcd}
\end{equation}
% If $\Sigmaw=\D/H$ is the $\T$-groupoid from
We defined a $\T$-groupoid $\Sigmaw$ associated to this extension in
\cite{ikrsw:jfa21}*{Proposition~3.2} as follows. Define
\[
  \D=\set{(\chi,z,\sigma)\in\hA\times\T\times\Sigma:
    \hp(\chi)=r(\sigma)}
\]
and let $H$ be the subgroupoid of $\D$ consisting of triples of the
form $(\chi,\overline{\chi(a)},a)$ for $a\in \A(\hp(\chi))$. Then $H$
is a normal subgroupoid of $\D$ and we can form the locally compact
Hausdorff groupoid $\Sigmaw:=\D/H$
(we use the notation $\Sigmaw$, rather than the notation 
$\widehat\Sigma$ of \cite{ikrsw:jfa21}, to avoid clashing with 
classical notational conventions when $\Sigma$ is a group, for 
example in Remark~\ref{rem-mackey}).

\begin{thm}
  \label{thm:pushout}
  Let $\Sigma$ be the extension of $\G$ by the group bundle $\A$ as in
  the diagram~(\ref{eq:formal-ext-intro}) and adopt the notation established
  above.  Let $f:\hA *\A\to \hA\times \T$ be the canonical map given by
  \begin{equation}
    \label{eq:canmap}
    f(\chi,a)=( \chi,{\chi(a)}).
  \end{equation}
  Then $\Sigmaw$ is properly isomorphic to the pushout
  $f_{*}(\hA\rtimes \Sigma)$. Moreover,
  \[
    \cs(\Sigma) \cong \cs(\hA\rtimes\G; f_{*}(\hA\rtimes \Sigma))
    \quad\text{and} \quad \cs_r(\Sigma) \cong \cs_r(\hA\rtimes\G;
    f_{*}(\hA\rtimes \Sigma)).
  \]
\end{thm}
\begin{proof}
  Theorem~\ref{thm:functor} implies that there is a unique (up to
  proper isomorphism) extension $f_*(\hA\rtimes \Sigma)$ of
  $\hA\rtimes \G$ by $\hA\times \T$ and a twist morphism that is
  compatible with $f$. In particular, $f_{*}(\hA\rtimes \Sigma)$ is a
  $\T$-groupoid. % , also called
  % a twist
  We get a natural map $g:\hA\rtimes\Sigma$ to $\Sigmaw$ given by
  $g(\chi,\sigma)=[\chi,1,\sigma]$, and the diagram
  \begin{equation}
    \label{eq:82}
    \begin{tikzcd}[row sep = 1ex, column sep = 2cm]
      \hA*\A\arrow[r,"\iota_{*}"] \arrow[dd,"f",swap]& \hA\rtimes
      \Sigma
      \arrow[dr,"p_{*}"] \arrow[dd,"g",swap]\\
      &&\hA\rtimes\G \\
      \hA\times \T\arrow[r,"i",swap]&\Sigmaw \arrow[ur,"j",swap]
    \end{tikzcd}
  \end{equation}
  commutes.  The proper isomorphism of $\Sigmaw$ with
  $f_{*}(\hA\rtimes \Sigma)$ follows from the uniqueness guaranteed by
  Theorem~\ref{thm:functor} and the final assertion follows from
  \cite{ikrsw:jfa21}*{Theorem~3.3}.
\end{proof}

It follows immediately that if $\Sigma$ is properly isomorphic to the
semidirect product $\SD(\A,\G)$, then
$[\hA\rtimes\Sigma] = [\hA\rtimes(\SD(\A,\G))] =[\SD(\A,(\hA\rtimes\G))]$ and hence $[\Sigmaw]$ is
trivial. Thus $\cs(\Sigma) \cong \cs(\hA\rtimes\G)$.

\begin{remark} \label{rem-mackey} 
  As mentioned in the introduction, the twist $\Sigmah$ appearing in Theorem~\ref{thm:pushout} is responsible for the Mackey obstruction of the classical normal subgroup analysis of  \cite{mac:am58}. Indeed, let us apply the theorem when $\Sigma$ is a
  locally compact group and $\A$ is a closed normal abelian subgroup.  Then $\Sigma$ and $\G=\Sigma/\A$ act on
  $\A$ by conjugation and give right actions on the space of characters $\hA$.
  The corresponding twist $\Sigmah$ 
  is the quotient of the 
  groupoid $(\hA\rtimes \Sigma)\times\T$ where $(\chi,a\sigma,\theta)$ is
  identified with $(\chi,\sigma,\theta\chi(a))$ for all $a\in \A$.  We
  let $[\chi,\sigma,\theta]$ be the class of $(\chi,\sigma,\theta)$ in
  $\Sigmah$.   If $\chi\in\hA$, then let $\Sigma(\chi)$ and $\G(\chi)$ be
  the stabilizers at $\chi$ for the actions on $\hA$, and let
  $\Sigmah(\chi)$ be the isotropy group of $\Sigmah$ at $\chi$.  We observe that
  $\Sigmah(\chi)$, up to an obvious identification, is the pushout of the group extension
  \begin{equation}
    \label{eq:26}
    \begin{tikzcd}
      \A\arrow[r]&\Sigma(\chi)\arrow[r]&\G(\chi)
    \end{tikzcd}
  \end{equation}
by the homomorphism $\chi:\A\to\T$. Indeed, this pushout $\chi_*(\Sigma(\chi))$ is the quotient of
$\Sigma(\chi)\times \T$ by the equivalence relation identifying
$(a\sigma,\theta)$ with $(\sigma,\theta\chi(a))$ for all $a\in \A$. Thus we just identify $[\chi,\sigma,\theta]\in\Sigmah(\chi)$ with $[\sigma,\theta]\in\chi_*(\Sigma(\chi))$. The class of $\Sigmah(\chi)$ in $H^{2}(\G(\chi),\T)$ is the classical Mackey obstruction. More precisely, let $L$ be an irreducible unitary representation of $\Sigma$. According to Theorem~\ref{thm:pushout}, we may view it as a representation of the twisted groupoid $(\hA\rtimes \G,\Sigmah)$. Its restriction to $\hA$ defines a measure class which is invariant and ergodic under the action of $\G$. If this measure class is transitive, which will be always the case if  $\A$ is regularly embedded, then we have a representation of a twisted transitive measured groupoid $(O\rtimes \G,\Sigmah\restr{O})$, where $O\subset\hA$ is an orbit of the action and $\Sigmah\restr{O}$ is the reduction of $\Sigmah$ to $O$. We pick $\chi\in O$. Since the $(\Sigmaw(\chi),\Sigmah\restr{O})$-groupoid equivalence $\Sigmah_{O}^{\chi}$ is compatible with the twists in the sense of \cite[D\'efinition 5.3]{ren:jot87},
%[J. Renault, Repr\'esentation des produits crois\'es d'alg\`ebres de groupo\"ides, J. Operator Theory 18 (1987), 67--97], 
it implements a bijective correspondence between the unitary representations of $(O\rtimes \G,\Sigmah\restr{O})$ and those of $(\G(\chi),\Sigmah(\chi))$. Therefore $L$ is given by an irreducible unitary representation of the twisted group $(\G(\chi),\Sigmah(\chi))$. 
\end{remark}

\begin{example}\label{ex:abelian group} 
  Let $H$ be a locally compact abelian group and let $A \subset H$ be
  a closed subgroup.  Then applying the above theorem with
  $\Sigma = H$ and $\A = A$, we conclude that $\Sigmaw$ %% $f_{*}(\hat{A}\rtimes H)$
 is a bundle of abelian groups over $\Sigmaw^{(0)} \cong \hat{A}$ where each fiber is
  an extension of $H/A$ by $\T$.  Each of these extensions is
  abelian because $H$ is abelian (and the action of $H$ on $\hat{A}$ is trivial).  
  Hence, each extension is determined by a symmetric $\T$-valued Borel $2$-cocycle 
  and any such $2$-cocycle is necessarily trivial by \cite{kle:ma65}*{Lemma~7.2}.
  But the twist is not trivial in general: for example, if $H=\R$ and
  $A=\Z \le \R$, then triviality of the twist would imply $\cs(\R)\cong C_{0}(\T\times\Z)$, 
  which is nonsense.
\end{example}

  \begin{example}[Generalized Twists] \label{ex-gen-twist}
    We now consider the case where $A$ is a locally compact abelian group,
    $\A = \goo \times A$, and $\G$ acts on $\A$ by
translation on the first factor.   Since this simply gives us a twist
when $A=\T$, we will say that $\Sigma$ is a \emph{generalized twist}
in this case.  Note that even for twists,
$\Sigma$ need not be a trivial extension.  Generalized twists were
studied in \cite{iksw:jot19}.

View $\hA := \hat{A} \times \goo$ as a locally compact space. (We put the 
factor of $\goo$ on the right, as a reminder that we are thinking of
$\hat{A}$ as a space rather than as a group, and to line up with the 
natural identification of $\hA * \A$ with $\hat{A} \times \goo \times A$, 
which we make without further comment). Then 
$\G$ acts on the second factor of $\hA$.
This means we can replace $\hA\rtimes\G$ and $\hA\rtimes \Sigma$ with
the products $\hat{A}\times \G$ and $\hat{A} \times \Sigma$,
respectively. Under these identifications, Equation~\eqref{eq:canmap} becomes
$f(\chi, u, a) = (\chi, u, {\chi(a)})$.  
Moreover we may
assume that the Haar system $\beta$  on $\A = \goo \times A$ 
is constant in the sense that there is a fixed Haar measure $\mu$ on $A$
such $\beta^u = \mu$ for all $u \in \goo$.
\end{example}

If $\chi \in \hat A$, then we get a $\G$-equivariant map
$\fchi : \goo \times A \to \goo \times \T$ given by
$\fchi(u, a)=(u, {\chi(a)})$. Thus we can form the pushout
$\fchiss$ so that
\begin{equation}
  \label{eq:4}
  \begin{tikzcd}[column sep= 3cm, row sep=2ex]
    \goo\times A \arrow[r,"\iota"] \arrow[dd,"\fchi"]
    &\Sigma \arrow[rd,"p"] \arrow[dd,"\fchi_{*}"] \\
    &&\G \\
    \goo\times\T \arrow[r,"\iota'"] & \fchiss
    \arrow[ru,"p'"]
  \end{tikzcd}
\end{equation}
commutes.  Then $\cs(\G;\fchiss)$ is the
completion of $C_{c}^{\chi}(\Sigma)$ consisting of functions
$g\in C_{c}(\Sigma)$ such that
$g(\iota(r(\sigma),a)\sigma) =\chi(a) g(\sigma)$ with the $*$-algebra
structure discussed at the beginning of this section.

\begin{prop} \label{prop:bundle} Let $\Sigma$ be a generalized twist
  as in Example~\ref{ex-gen-twist}. For $\chi \in \hat A$, let
  $\fchi:\goo \times A \to \goo \times \T$ and $\fchiss$ be the $\G$-equivariant
  map and $\T$-groupoid defined above. Then with notation as above,
  \begin{equation}
    \label{eq:1}  
    \cs(\Sigma) \cong \cs(\hat{A} \times \G; f_{*}(\hat{A} \times
    \Sigma)) 
  \end{equation}
  and $\cs(\Sigma)$ is the section algebra of an upper-semicontinuous
  $C^*$-bundle over $\hat{A}$ with fiber at $\chi \in \hat{A}$
  isomorphic to $\cs(\G; \fchiss)$.
\end{prop}
\begin{proof}
  The isomorphism in \eqref{eq:1} comes from
  Theorem~\ref{thm:pushout}.

  The map $p: \hat{A} \times \goo \to \goo$ is continuous and satisfies
  $p\circ s=p\circ r$ so that $f_{*}(\hat A\times\Sigma)$ is a
  groupoid bundle over $\hat{A}$ as in
  Appendix~\ref{sec:bundles-twists}.   Hence we can invoke
  Proposition~\ref{prop-jean2} to see that $\cs(\hat{A} \times \G;
  f_{*}(\hat{A} \times \Sigma))$ is isomorphic to the section algebra
  of an upper-semicontinuous \cs-bundle over $\hat A$.    Since we can identify
  $f_{*}(\hat A\times\Sigma)(\chi)$ with $\fchiss$ and
  $(\hat A\times\G)(\chi)$ with $\G$, the result follows. 
\end{proof}

%\begin{remark}\label{rmk:compact-case}
\begin{prop}\label{prop:compact-case}
 With notation as in Example~\ref{ex-gen-twist}, suppose that   $A$ compact.
%  In the case of a generalized twist with
Then the dual $\hat{A}$ is discrete and
  \begin{equation}
    \label{eq:28}
    \cs(\Sigma) \cong \bigoplus_{\chi \in \hat{A}} \cs(\G; \fchiss) 
    \quad\text{and}\quad
    \cs_r(\Sigma) \cong \bigoplus_{\chi \in \hat{A}} \cs_r(\G; \fchiss).
  \end{equation}
\end{prop}
 
\begin{proof}
 To prove the first isomorphism, note that by Proposition~\ref{prop-jean2}
\[
\cs(\Sigma) \cong \cs(\hat{A} \times \G; f_{*}(\hat{A} \times \Sigma))
 \]
is a $C_0(\hat{A})$-algebra.
That is, letting $ZM(\cs(\Sigma))$ denote the center of $M(\cs(\Sigma))$, 
there is a $\sigma$-unital *-homomorphism $\rho: C_0(\hat{A}) \to ZM(\cs(\Sigma))$.
Since $\hat{A}$ is discrete, %by the isomorphism \eqref{eq:1}
the images of the characteristic functions of singleton sets under $\rho$ give rise to
a family $\{ q_\chi \}_{\chi \in \hA}$ of mutually orthogonal central projections in 
  $M(\cs(\Sigma))$ which sum to unity in the strict topology.  
  Moreover, the summands coincide with the fibers of the  upper-semicontinuous
  $C^*$-bundle over $\hat{A}$ given in Proposition~\ref{prop:bundle}  and hence
  \[
  q_\chi \cs(\Sigma)q_\chi  =  q_\chi \cs(\Sigma) \cong \cs(\G; \fchiss).
  \]
  for all $\chi \in \hat{A}$.
  
  For the second isomorphism, let $\pi: \cs(\Sigma) \to \cs_r(\Sigma)$
  be the canonical quotient map. An argument like that of the preceding paragraph
  using the family $\{ \pi(q_\chi) \}_{\chi \in \hA}$ of mutually orthogonal central
  projections in  $M(\cs_r(\Sigma))$ gives $C^*_r(\Sigma) \cong \bigoplus_{\chi \in \hA}
  \pi(q_\chi) C^*_r(\Sigma)$. Lemma~\ref{lem:inv-clopen} gives 
  $\pi(q_\chi) \cs_r(\Sigma) \cong \cs_r(\G; \fchiss)$, and the result follows.
\end{proof}
 
\begin{remark}\label{rmk:t-case}
  If $A=\T$ and $\Sigma$ is a twist, then $\hat{A} = \Z$, and we have
  $[f^n_{*}(\Sigma)] = n[\Sigma]$ for $n \in \Z$.
  % Hence,
  % \[
  %   \cs(\Sigma) \cong \bigoplus_{n \in \Z} \cs(\G; n_*(\Sigma)).
  % \]
  It follows that the central summand corresponding to $n = 1$
  is isomorphic to $\cs(\G; \Sigma)$ and thus there is central projection
  $q = q_1 \in M(\cs(\Sigma))$ such that
  \[
    \cs(\G; \Sigma) \cong q\cs(\Sigma)    \quad\text{and}\quad
    \cs_r(G; \Sigma) \cong q\cs_r(\Sigma)
  \]
% A similar assertion holds for the reduced $C^*$-algebras.
  % Note that if $n$ is a positive integer we have
  % $n_*(\Sigma) = \Sigma *_\G \Sigma *_\G \cdots *_\G \Sigma$ where
  % there are $n$-factors.
\end{remark}

Now suppose that $\G=\goo$ so that $\Sigma = \A$ is itself an abelian
group bundle 
regarded as a groupoid with unit space $\goo $ and let $\Twist$ be a
$\T$-twist over $\A$. 
Then since $\A$ is amenable $\cs(\A; \Twist) = \cs_r(\A; \Twist)$ 
(see, for example \cite{simwil:ijm13}*{Thm 1}).
We shall say that such a twist is \emph{abelian} if $\Twist$ is also an abelian
group bundle---that is if $\Twist(u)$ is abelian for each $u \in \goo$.
Then $\Twist$ is abelian if and only if $\cs(\Twist)$ is abelian and
in that case $\cs(\Twist) \cong C_0(\hat\Twist)$.  Arguing as in
Example~\ref{ex:abelian group}, we see that such extensions must be pointwise 
trivial but need not be globally trivial.  If $\Twist$ is determined by a
continuous $\T$-valued 2-cocycle $c$, then $\Twist$ is abelian if and
only if $c$ is symmetric (cf., \cite[Lemma~3.5]{dgnrw:jfa20}).
Suppose now that $\Twist$ is abelian. For $n \in \Z$, let 
$V_n := \{ \chi \in \hat\Twist : \chi(z,u) = z^n\,\text{ for all }z\in
\T\text{ and }u\in\goo \}$. Then $\cs(\Twist) \cong C_0(\hat\Twist)$ 
decomposes as a direct sum with summands of the form $C_0(V_n)$.  
Note that each $V_n$ is clopen.
% Then by dualizing the central extension
% \[
%   X \times \T \xrightarrow{\jmath} \Twist \xrightarrow{\eta} \A
% \]
% we obtain another extension of abelian group bundles
% \[
%   \hA \xrightarrow{\hat\eta} \hat\Twist \xrightarrow{\hat\jmath} X
%   \times \Z.
% \]
% By the continuity of $\hat\jmath$,
% \[
%   \hat\jmath^{-1}(X \times \{ n \}) = \{ \chi \in \hat\Twist :
%   \chi(\jmath((\hat{p}(\chi), z))) = z^n \text{ for all } z \in \T
%   \}
%%%   \{ \chi \in \hat\Twist : \chi(z\xi) = z^n\xi \text{ for all } z
%%%   \in \T, \xi \in \Twist(\hat{p}(\chi)) \}
% \]
% is a clopen subset of $\hat\Twist$ for every $n \in \Z$.
The projection $q$ in Remark~\ref{rmk:t-case} may then be
identified with the characteristic function of $U_\Twist := V_{1}$
and hence 
% by Remark~\ref{rmk:t-case}
\[
  \cs(\A; \Twist) \cong q\cs(\Twist) \cong
  C_0(U_\Twist).  %%= qC_0(\hat\Twist)q
\]
See \cite[Section 3]{dgn:xx20} for a related construction.

In the case that $\Twist \cong \T \times \A$ and thus
$\hat\Twist \cong \Z \times \hat{\A}$, we have
$U_\Twist \cong \{ 1 \} \times \hat{\A} \cong \hat{\A}$.

We return now to the more general situation where $\Sigma$ is a unit
space fixing extension of $\G$ by the group bundle $\A$ as in the
diagram~(\ref{eq:formal-ext-intro}) from the introduction.  Suppose
that, in addition, $\Omega$ is a $\T$-groupoid extension of $\Sigma$
\begin{equation}
  \label{eq:omega-ext}\tag{\star}
  \begin{tikzcd}[column sep=3cm]
    \goo \times \T \arrow[r,"\tilde\iota"] \arrow[dr,shift left, bend
    right = 15] \arrow[dr,shift right, bend right = 15]&\Omega
    \arrow[r,"\tilde{p}", two heads] \arrow[d,shift left]
    \arrow[d,shift right]&\Sigma \arrow[dl,shift left, bend left = 15]
    \arrow[dl,shift right, bend left = 15]
    \\
    &\goo&
  \end{tikzcd}
\end{equation}
such that $\Twist_\Omega := \tilde{p}^{-1}(\A)$,  %% = \Omega|_{\A}$, 
its restriction to $\A$, is an abelian group bundle over $\goo$.  We may
thus regard $\Omega$ as an extension of $\G$ by $\Twist_\Omega$.  We
assume that $\A$, $\Sigma$ and $\G$ are endowed with Haar systems that
satisfy \eqref{eq:harhar}, the Haar system in $\goo \times \T$ is given
by the Haar measure on $\T$, and the Haar system on $\Omega$ is the
one naturally defined by the Haar systems on $\goo \times \T$ and
$\Sigma$. To declutter notation a little, we write $\hTwistOmega$ for the dual 
bundle $(\Twist_\Omega)^\wedge$.
%% $\A \times \T$ .  %%Since the above extension $\Omega$ is central
%% \widehat{\Omega|_{\A}}

\begin{cor}\label{cor:ext-abelian-bundle}
  With notation as above let
  $f : \hTwistOmega* \Twist_\Omega \to
  \hTwistOmega \times \T$ be given by
  $f(\chi, a)=( \chi,{\chi(a)})$.  
  Then
  \begin{align*}
    \cs(\Omega)&\cong \cs( \hTwistOmega\rtimes\G; f_{*}(
                 \hTwistOmega \rtimes \Omega))
                 \quad\text{and}  \\ 
    \cs_r(\Omega) &\cong  \cs_r( \hTwistOmega\rtimes\G;
                    f_{*}( \hTwistOmega \rtimes
                    \Omega)). 
  \end{align*}
  % An analogous result holds for the reduced $\cs$-algebra.
\end{cor}
\begin{proof}
  This follows immediately from Remark~\ref{rmk:t-case}, the above discussion, and
  Theorem~\ref{thm:pushout} with $\Twist_\Omega$ in place of $\A$.
\end{proof}
By arguing as in Remark~\ref{rmk:t-case} and
Corollary~\ref{cor:ext-abelian-bundle} we may conclude that
$\cs(\Sigma; \Omega)$ is isomorphic to the corner associated to the
central projection $q_\Omega$ in
\[
  M(\cs( \hTwistOmega\rtimes\G; f_{*}(
  \hTwistOmega \rtimes \Omega)))
\]
corresponding to the characteristic function of
\[
  U_\Omega := U_{\Twist_\Omega} \subset \hTwistOmega =
  ( \hTwistOmega\rtimes\G)^{(0)}.
\]
Observe that $U_\Omega$ is an invariant clopen set under the action of
both $\G$ and $\Omega$ and thus both groupoids act on $U_\Omega$.

\begin{cor} \label{cor:twisted-extensions}
  % let
  % $\iota : U_\Omega \xhookrightarrow {}\hTwistOmega$
  % denote the inclusion map.
  With notation as above define
  $g: U_\Omega*\Twist_\Omega \to U_\Omega \times \T$ by
  $g(\chi, a)=( \chi,{\chi(a)})$.   Then
  \[
    \cs(\Sigma; \Omega) \cong \cs(U_\Omega\rtimes\G;
    g_{*}(U_\Omega\rtimes \Omega)) \quad\text{and} \quad \cs_r(\Sigma;
    \Omega) \cong \cs_r(U_\Omega\rtimes\G; g_{*}(U_\Omega\rtimes
    \Omega)).
  \]
\end{cor}
\begin{proof}
  Observe that
  \[
    \big(\hTwistOmega\rtimes\G\big)_{U_\Omega} \cong
    U_\Omega\rtimes\G \quad \text{and} \quad
    \big(\hTwistOmega\rtimes\Omega\big)_{U_\Omega}
    \cong U_\Omega\rtimes\Omega.
  \]
  For
  $(\chi, a) \in U_\Omega*\Twist_\Omega \subset
  \hTwistOmega *\Twist_\Omega$,
  \[
    f(\chi, a) = (\chi,{\chi(a)}) = g(\chi, a) \in U_\Omega
    \times \T
  \]
  Therefore,
  \[
    \big(f_{*}( \hTwistOmega \rtimes
    \Omega)\big)_{U_\Omega} \cong g_{*}(U_\Omega\rtimes \Omega).
  \]
  Hence, by Remark~\ref{rmk:t-case} and
  Corollary~\ref{cor:ext-abelian-bundle}
  \begin{align*}
    \cs(\Sigma; \Omega)
    &\cong q_\Omega \cs\big(
      \hTwistOmega\rtimes\G; f_{*}(
      \hTwistOmega \rtimes \Omega)\big)
      q_\Omega\\ 
    &\cong
      \cs\big((\hTwistOmega\rtimes\G)_{U_\Omega};
      (f_{*}( \hTwistOmega \rtimes
      \Omega))_{U_\Omega}\big)\\ 
    &\cong \cs\big(U_\Omega\rtimes\G;
      g_{*}(U_\Omega\rtimes \Omega)\big). 
  \end{align*}
  The case for the reduced $\cs$-algebras follows by a similar
  argument.
\end{proof}

Recall that an \'etale groupoid $\G$ is said to be \emph{effective}  if the interior of the isotropy groupoid is $\goo$
and  \emph{topologically principal}  if  the set of points with trivial isotropy is dense in $\goo$.   
These notions are equivalent if the  \'etale groupoid $\G$ is second countable (see \cite[Lemma~3.1]{bcfs:sf14}).
%Recall that an \'etale groupoid $\G$ is called effectiveCartan subalgebras are defined in  \cite[Definition 5.1]{ren:irms08}.
The above corollary allows us to generalize~\cite[Theorem 4.6]{ikrsw:jfa21} (see also \cite[Theorem 5.8]{dgnrw:jfa20} and
\cite[Theorem 4.6]{dgn:xx20}).

\begin{cor}\label{cor:twisted-cartan}
  With notation as above, suppose that $\G$ is  \'etale and that the
  action groupoid 
  $U_\Omega\rtimes \G$ is second countable and effective.  %% topologically principal. %
  %% let $\Omega|_{\A}$ denote the restriction of $\Omega$ to $\A$
  %% (note that $\Omega|_{\A} \cong \T \times \A$) and
  Then the image of $\cs_r(\A, \Twist_\Omega)$ under the natural
  embedding into $\cs_r(\Sigma; \Omega)$ is a Cartan subalgebra with
  Weyl twist $g_{*}(U_\Omega \rtimes \Omega)$.
\end{cor}
\begin{proof}
 This follows from Corollary~\ref{cor:twisted-extensions} and \cite[Theorem 5.2]{ren:irms08}.
\end{proof}

%% if its action on $\hat\Twist$ is effective or, in particular

\begin{example}\label{ex:biccharacter}
  Let $H$ be a discrete abelian group and let $E$ be a $\T$-twist over
  $H$---that is, a central extension by $\T$. Since $H$ is discrete,
  %we may assume that $E$ is given by a $\T$-valued 2-cocycle.
  % (see \cite[Example
  % 12]{kum:cjm86}). %Since $H$ is discrete, there is a cocycle
  there is a $\T$-valued skew-symmetric bicharacter $\varpi$ on
  $H$ and a set of generating unitaries $\{ u_h \mid h \in H \}$ in
  $\cs(H; E)$ such that for all $g, h \in H$
  \[
    u_gu_h = \varpi(g, h)u_hu_g.
  \]
  By \cite[Lemma 7.2]{kle:ma65} the extension $E$ is trivial if and
  only if $\varpi(g, h) = 1$ for all $g, h \in H$.  Let $A$ be a
  subgroup of $H$ which is maximal amongst subgroups on which 
  $\varpi(\cdot, \cdot)$ is identically $1$.  It is shown in
  \cite[Example 1.12]{kum:cjm86} that the \cs-subalgebra $B$ generated
  by $\{ u_a \mid a \in A \}$ is a diagonal subalgebra of $\cs(H; E)$.
  We now show that this also follows from Corollary
  \ref{cor:twisted-cartan} with $\Sigma := H$, $\A := A$, $\G = H/A$
  and $\Omega := E$.

  Since the restriction of $\varpi$ to $A$ is trivial the extension
  $E$ is trivial on $A$ and thus $\Twist$ is trivial as a $\T$-twist.
  Hence, $B \cong \cs(A)$ and $U_\Twist \cong \hat{A}$.  There is a
  continuous homomorphism $\varpi_A: H \to \hat{A}$ such that for all
  $h \in H$, $a \in A$
  \[
    (\varpi_A(h))(a) = \varpi(h, a).
  \]
  Moreover, $A = \ker\varpi$ and thus $\varpi$ induces an injection
  $H/A \to \hat{A}$.  The action of $H/A$ on $\hat{A}$ is then given
  by translation and, hence, is free.  Since $H/A$ is \'etale and its
  action on $U_\Omega \cong \hat{A}$ is principal, the image of
  $\cs_r(\A, \Twist_\Omega) \cong \cs(A)$ under the natural embedding
  into $\cs_r(\Sigma; \Omega) = \cs(H; E)$ is a diagonal subalgebra.  
\end{example}

\subsection{Extensions by 2-cocycles}
\label{sec:twists-cocycles}
Extensions associated to groupoid $2$-cocycles yield some nice
applications of the pushout construction. For convenience, we review
the basics here. (For more details, see
\cite{iksw:jot19}*{Appendix~A}.)  Assume that $p_{\A}:\A\to\goo$ is a
$\G$-bundle. As before we write $\A(u)$ for $p_{\A}^{-1}(u)$ for
$u\in\goo$. Assume that $\phi:\G^{(2)}\to \A$ is a continuous
normalized 2-cocycle. That is,
$\phi(\gamma_1,\gamma_2)\in \A(r(\gamma_1))$ for all
$(\gamma_1,\gamma_2)\in \G^{(2)}$,
$\phi(\gamma_0,\gamma_1)+\phi(\gamma_0\gamma_1,\gamma_2)=\gamma_0\cdot
\phi(\gamma_1,\gamma_2)+\phi(\gamma_0,\gamma_1\gamma_2)$ for all
$(\gamma_0,\gamma_1),(\gamma_1,\gamma_2)\in\G^{(2)}$, and
$\phi(\gamma,u)=\phi(u,\gamma)=0_u$ for all $\gamma\in \A(u)$ and
$u\in\goo$.  Then the extension $\Sigma_\phi$ of $\G$ by $\A$
determined by $\phi$ is obtained by giving the fibered product $\A*\G$
the groupoid structure where
$(a_1,\gamma_1)(a_2,\gamma_2)=(a_1 +\gamma_1\cdot
a_2+\phi(\gamma_1,\gamma_2),\gamma_1\gamma_2)$ if
$(\gamma_1,\gamma_2)\in \G^{(2)}$ and
$(a,\gamma)^{-1}=(-\gamma^{-1}\cdot
a-\phi(\gamma^{-1},\gamma),\gamma^{-1})$.  We exhibit $\Sigma_\phi$ as
an extension of $\G$ by $\A$ via $i(a)=(a,p_{\A}(a))$ and
$p(a,\gamma)=\gamma$.

%The two main classes of examples that we have in mind are the following:
\begin{example}
  If $\A = \goo \times A$ is the trivial bundle (with trivial action),
  then an $\A$-valued cocycle is given by a continuous $A$-valued
  2-cocycle $\sigma$ on $\G$ via the formula 
  $\varphi(\gamma_1, \gamma_2) = (\sigma(\gamma_1,\gamma_2), r(\gamma_1))$.
\end{example}

\begin{example}\label{ex:astrid-jon}
  Let $\phi$ be a continuous normalized $\T$-valued 2-cocycle and let
  $\Sigma_\phi$ be the $\T$-twist associated to $\phi$.  Then by
  Proposition~\ref{prop:compact-case} and Remark~\ref{rmk:t-case}, and
  the fact that $\Sigma_{\phi^n} \cong n_*(\Sigma_\phi)$ for all
  $n \in \Z$, we have
  \[
    \cs(\Sigma_\phi) \cong \bigoplus_{n \in \Z} \cs(\G;
    \Sigma_{\phi^n}).
  \]
  This recovers \cite[Theorem 3.2]{broanh:pams14}.
\end{example}

\begin{example}[Transformation groupoids]\label{ex:transgr}
  Let $\G$ be a groupoid acting on the right of a locally
  compact Hausdorff space $X$. Recall that the transformation groupoid
  $X\rtimes \G$ is obtained by endowing the fibered product $X*\G$
  with the groupoid operations
  $(x,\gamma_1)(x\cdot \gamma_1,\gamma_2)=(x,\gamma_1\gamma_2)$ if
  $(\gamma_1,\gamma_2)\in\G^{(2)}$ and
  $(x,\gamma)^{-1}=(x\cdot \gamma,\gamma^{-1})$.

  Assume that $\phi:\G^{(2)}\to \A$ is a 2-cocycle as above. Then one
  can define a natural 2-cocycle
  $\tilde{\phi}:(X\rtimes \G)^{(2)}\to X*\A$ via
  $\tilde\phi((x,\gamma_1),(x\cdot\gamma_1,\gamma_2))=
  (x,\phi(\gamma_1,\gamma_2))$. The extension
  $\Sigma_{\tilde{\phi}}$ of $X\rtimes \G$ defined by $\tilde{\phi}$
  is isomorphic to the extension $X\rtimes \Sigma_\phi$, where
  $\Sigma_\phi$ is the extension of $\G$ defined by $\phi$. To see
  this, note that
  $\Sigma_{\tilde{\phi}}=\set{ \bigl( (x,a),(x,\gamma)\bigr) : x \in X, a \in \A^x, \gamma \in \G^x}$ with
  the operations
  \[\bigl(
    (x,a_1),(x,\gamma_1) \bigr)\bigl( (x\cdot \gamma_1,a_2),(x\cdot
    \gamma_1,\gamma_2) \bigr)=\bigl(
    (x,a_1+\gamma_1a_2+\phi(\gamma_1,\gamma_2)),(x,\gamma_1\gamma_2)
    \bigr)\] and
  \[\bigl( (x,a),(x,\gamma) \bigr)^{-1}=\bigl( (x\cdot
    \gamma,-\gamma^{-1}a-\phi(\gamma^{-1},\gamma)),(x\cdot
    \gamma,\gamma^{-1}) \bigr).
  \] 
  On the other hand,  
  $X\rtimes \Sigma_\phi=\set{(x,(a,\gamma)) : x \in X, a \in \A^x, \gamma \in \G^x}$ 
  with the operations
  \begin{equation}
    \label{eq:10}
    (x,(a_1,\gamma_1))(x\cdot
    \gamma_1,(a_2,\gamma_2))=
    (x,(a_1+\gamma_1a_2+\phi(\gamma_1,\gamma_2),\gamma_1\gamma_2)) 
  \end{equation}
  and
  \begin{equation}
    \label{eq:13}
    (x,(a,\gamma))^{-1}=(x\cdot \gamma,(-\gamma^{-1}\cdot
    a-\phi(\gamma^{-1},\gamma),\gamma^{-1})).
  \end{equation}
  Therefore the map $V:\Sigma_{\tilde{\phi}}\to X\rtimes \Sigma_\phi$
  defined by $V\bigl( (x,a),(x,\gamma) \bigr)=(x,(a,\gamma))$ is a
  groupoid isomorphism.
  
\end{example}

Suppose that $p_{\B}:\B\to\goo$ is another abelian $\G$-bundle and
that $f : \A\to \B$ is an equivariant map such that
$f|_{\A(u)}:\A(u)\to \B(u)$ is a continuous group homomorphism for all
$u\in \goo$. There is a $\B$-valued 2-cocycle
$f_*(\phi) : \G^{(2)}\to \B$ given by
$f_*(\phi) (\gamma_1,\gamma_2)=f(\phi(\gamma_1,\gamma_2))$.
\begin{lemma}\label{lem:push2coc}
  Let $\Sigma_{f_*(\phi) }$ be the extension of $\G$ by $\B$
  determined by $f_*(\phi) $. Then $f_*\Sigma_\phi$ is properly
  isomorphic to $\Sigma_{f_*(\phi)}$.
\end{lemma}
\begin{proof}
  Define $g:\Sigma_{\phi}\to \Sigma_{f_{*}\phi}$ by
  $g(a,\gamma)=(f(a),\gamma)$.  The diagram
  \begin{equation}
    \label{eq:77}
    \begin{tikzcd}[row sep = 1ex, column sep = 2cm]
      \A \arrow[dd,"f",swap] \arrow[r,"i"] & \Sigma_{\phi}
      \arrow[dd,"g"] \arrow[rd,"p"]
      \\
      &&\G \\
      \B\arrow[r,"i",swap] & \Sigma_{f_{*}(\phi)} \arrow[ru,"p",swap]
    \end{tikzcd}
  \end{equation}
  commutes.  Therefore the lemma follows from
  Theorem~\ref{thm:functor}.
\end{proof}

\subsection{The $\T$-groupoid defined by a 2-cocycle}
\label{sec:t-groupoid-2}

We continue to assume the setting from
Section~\ref{sec:twists-cocycles}: $\A$ is an abelian $\G$-bundle,
$\phi:\G^{(2)}\to \A$ is a 2-cocycle, and $\Sigma_\phi$ is the
extension defined by $\phi$. Then, as in Example \ref{ex:transgr}
there is a 2-cocycle
\[
  \tilde{\phi}:(\hA\rtimes \G)^{(2)}\to \hA \ast\A
\]
defined by
\begin{align}
  \label{eq:phitilde}
  \tilde{\phi}\bigl(
  (\chi,\gamma_1),(\chi\cdot \gamma_1,\gamma_2)
  \bigr)=\bigl(\chi, \phi(\gamma_1,\gamma_2) \bigr)
\end{align}
if $(\gamma_1,\gamma_2)\in\G^{(2)}$. Therefore we can identify
$\hA\rtimes \Sigma_\phi$ with $\Sigma_{\tilde{\phi}}$, the extension
of $\hA\rtimes\G$ determined by $\tilde{\phi}$.  Consider the
2-cocycle
$\phiw:=f_*\tilde{\phi}:(\hA\rtimes \G)^{(2)}\to \hA\times \T$ defined
via 
\begin{equation}
  \label{eq:phihat}
  \phiw\bigl(
  (\chi,\gamma_1),(\chi,\gamma_2)
  \bigr)=\bigl(    \chi, {\chi(\phi(\gamma_1,\gamma_2))} \bigr).
\end{equation}
Lemma \ref{lem:push2coc}~and~Theorem \ref{thm:pushout} imply that
$\Sigmaw_\phi$ is isomorphic to the $\T$-groupoid defined by $\phiw$
and $C^*(\Sigma_\phi)$ is isomorphic to
$C^*(\hA\rtimes\G;\Sigma_{\phiw})$.

\begin{example}
  The following example was studied in \cite{iksw:jot19}. Let $X$ be a
  second-countable locally compact Hausdorff space, and $G$ a
  second-countable locally compact abelian group. Let $\uG$ denote the
  sheaf of germs of continuous $G$-valued functions on $X$, and let
  $c\in Z^2(\U,\uG)$ be a normalized \v{C}ech two cocycle for some
  locally finite cover $\U=\{U_i\}_{i\in I}$ of $X$ by precompact open
  sets.  The blow-up groupoid $\G_\U$ with respect to the natural map
  from $\bigsqcup_{i}U_i$ into $X$ is
  \[
    \G_\U=\{(i,x,j)\,:\,x\in U_{ij}:=U_i\cap U_j\}
  \]
  with $(i,x,j)(j,x,k)=(i,x,k)$ and $(i,x,j)^{-1}=(j,x,i)$. As noted
  in \cite{iksw:jot19}*{Remark~3.3}, the \v{C}ech 2-cocycle $c$
  defines a groupoid 2-cocycle $\phi_c:\G_\U^{(2)}\to G$ via
  \[
    \phi_c\bigl( (i,x,j),(j,x,k) \bigr)=c_{ijk}(x).
  \]
  Let $\Sigma_c$ be the extension of $\G_\U$ by the 2-cocycle
  $\phi_c$. Define
  \[
    \phiw:\bigl( (\hat{G}\times \bigsqcup_{i}U_i)\rtimes \G_\U
    \bigr)^{(2)}\to \T\times \hat{G}\times \bigsqcup_{i}U_i
  \]
  by
  \[
    \phiw\bigl( (\tau,(i,x,j)),(\tau,(j,x,k)) \bigr)=\bigl(
    \overline{\tau(c_{ijk}(x))},\tau \bigr)
  \]
  for $\tau\in \hat{G}$ and
  $\bigl( (i,x,j),(j,x,k) \bigr)\in (\G_\U)^{(2)}$. % For brevity, we
  % did not write $(i,x,i)$ and $(j,x,j)$ in the above
  % formula. \todo{But both appear in the formula?}
  Then $\phiw$ is a
  groupoid $2$-cocycle, and the pushout groupoid $\Sigmaw$ is
  isomorphic to the $\T$-groupoid that is the extension of
  $(\hat{G}\times \bigsqcup_{i}U_i)\rtimes \G_\U$ defined by $\phiw$.

  Let $\V=\{\hat{G}\times U_i\}_{i\in I}$ be the locally finite cover
  of $\hat{G}\times X$, let $\uT$ be the sheaf of germs of continuous
  $\T$-valued functions, and define $\nu^c  = \set{\nu^c_{ijk}}\in Z^2(\V, \uT)$
  by
  \[
    \nu^c\bigl((\tau,(i,x,j)),(\tau,(j,x,k))\bigr)=
    \overline{\tau(c_{ijk}(x))}.
  \]
  Then the 2-cocycle $\phiw$ is defined by the \v{C}ech 2-cocycle
  $\nu^c\in Z^2(\V,\uT)$.

  That is, $\nu^c$ is the normalized 2-cocycle considered in
  \cite{iksw:jot19}*{Equation (3.4)}.  Hence the generalized
  Raeburn--Taylor $C^*$-algebra $A(\nu)$ studied in \cite{iksw:jot19}
  is isomorphic to the restricted $C^*$-algebra of the $\T$-groupoid
  defined by the 2-cocycle $\nu^c$.

  By \cite{iksw:jot19}*{Lemma~5.2}, $A(\nu)$ is a continuous-trace
  \cs-algebra with spectrum $\hat G\times X$ with Dixmier--Douady
  invariant $\delta(A(\nu))=[\nu^{c}]$.  For a concrete example,
  let $G=\Z$ and choose a \v Cech $2$-cocycle $c$ associated to any
  line bundle.
\end{example}
\begin{example}
  This example is an expansion of \cite{ikrsw:jfa21}*{Example 4.10}.
  Let $\Gamma=\Z$ act on $\T$ via rotation by $\alpha\in \Q$:
  $z\cdot k:=ze^{2\pi ik\alpha}$. If $\alpha=m/n$ with $m$ and $n$
  relatively prime, then $n\Z$ fixes the action. We have a
  short exact sequence of groups
  \begin{equation}
    \label{eq:nZ_Z_Zn}
    \begin{tikzcd}
      n\Z \arrow[r,"",hook] & \Z \arrow[r,"p",two heads] & \Z_n.
    \end{tikzcd}
  \end{equation}
  The action on $\T$ leads to an extension of groupoids
  \begin{equation}
    \label{eq:rational}
    \begin{tikzcd}
      n\Z\times \T \arrow[r,"i",hook] & \T\rtimes \Z
      \arrow[r,"\pi",two heads] & \T \rtimes \Z_n.
    \end{tikzcd}
  \end{equation}
  Thus, using the notation from the previous section,
  $\A=\T\times n\Z$, $\Sigma=\T\rtimes \Z$, and $\G=\T\rtimes
  \Z_n$. The $C^*$-algebra $C^*(\T\rtimes \Z)$ is the rational
  rotation $C^*$-algebra $\A_\alpha$ (see, for example,
  \cite{DeBr:84}). The groupoid $\D$ is the cartesian product
  $\T\times\T_n\times \T \times \Z$, where $\T_n=\T/\Z_n$ is the dual
  of $n\Z$. The extension $\Sigmaw$ is the quotient of $\D$ where we identify
  $(\omega,\chi,z,nl+k)$ with $(\omega, \chi^{nl} ,z,k )$. Therefore
  the rational rotation algebra $\A_\alpha$ is the completion of
  continuous functions $F$ on $\T\times \T_n\times\Z$ such that
  $F(\omega,\chi,nl+k)=\chi^{nl}F(\omega,\chi,k)$ for all $l\in \Z$.

  The extension $\Sigmaw$ is properly isomorphic to the one
  defined by a 2-cocycle. % and that we recover the description of $\A_\alpha$ from
  % \cite{DeBr:84}*{Proposition 1}
  Indeed, let $\sigma=e^{2\pi i\alpha}\in \T$ and view $\sigma$ as a
  character on $\Z$. Thus we can identify $\Z_n$ with $\sigma(\Z)$ and
  then the map $p$ in the short exact sequence \eqref{eq:nZ_Z_Zn}
  equals $\sigma$. Choose $s\in \Z$ such that $sm=1\,(\text{mod
  }n)$. Then the map $\tau:\Z_n\to \Z$ defined by $\tau(k)=sk$ defines
  a cross-section of $\sigma$. In particular, $\Z$ is properly
  isomorphic to the extension $n\Z\times_\omega \Z_n$ by a two cocycle
  $\omega:\Z_n\times \Z_n\to n\Z$ defined by $\tau$. Using the proof
  of \cite{iksw:jot19}*{Proposition~A.6},
  $\omega(\dotk_1,\dotk_2)=\tau(\dotk_1)+\tau(\dotk_2)-\tau(\dotk_1+\dotk_2)$.
  A quick computation shows that
  \[
    \omega(\dotk_1,\dotk_2)=
    \begin{cases}
      0 & \text{ if } \dotk_1+\dotk_2 <n\\
      ns & \text{ if } \dotk_1+\dotk_2\ge n,
    \end{cases}
  \]
  which recovers the 2-cocycle used in Step 2 of the proof of
  \cite{DeBr:84}*{Proposition~1}.

  The map $\underline{\tau}:\T\rtimes \Z_n\to \T\rtimes \Z$ defined
  by $\underline{\tau}(z,k)=(z,\tau(k))$ is a cross-section of the
  extension of the groupoids~\eqref{eq:rational}. Hence $\T\rtimes \Z$
  is properly isomorphic to the extension given by the 2-cocycle
  $\varphi\in Z^2(\T\rtimes \Z_n,\T\times n\Z)$ defined by
  $\varphi\bigl((w,\dotk_1),(w\cdot
  \dotk_1,\dotk_2)\bigr)=(w,\omega(\dotk_1,\dotk_2))$. The extension
  of the 2-cocycle $\varphi$ is
  $\Sigma_\varphi= \T\times n\Z\times \Z_n$ with operations
  $(w,nl_1,\dotk_1)(w\cdot
  \dotk_1,nl_2,\dotk_2)=(w,nl_1+nl_2+\omega(\dotk_1,\dotk_2),\dotk_1+\dotk_2)$
  and
  $(w,nl,\dotk)^{-1}=(w,-nl-\omega(-\dotk,\dotk),-\dotk)$. Following
  the proof of \cite{iksw:jot19}*{Proposition~A.6} the isomorphism
  between $\Sigma_\varphi$ and $\T\rtimes \Z$ is given by
  $(w,nl,\dotk)\mapsto (w,nl+\tau(\dotk))$.

  We have that $\hA\simeq \T_n\times \T$ and
  $\hA* \A\simeq \T_n\times \T\times n\Z$. The action of
  $\G=\T\rtimes \Z_n$ on $\hA$ is given via
  $(\chi,w)\cdot (w,\dotk)=(\chi,w\cdot
  k)=(\chi,w\sigma^k)$. Therefore we can identify $\hA\rtimes\G$ with
  $\T_n\times\T\rtimes \Z_n:= \{(\chi,w,\dotk)\in \T_n\times\T\times
  \Z_n\}$, where
  $(\chi,w,\dotk_1)\cdot (\chi,w\cdot
  k_1,\dotk_2)=(\chi,w,\dotk_1+\dotk_2)$ and
  $(\chi,w,\dotk)^{-1}=(\chi,w\cdot k,-\dotk)$.  Thus the 2-cocycle
  $\tilde{\phi}:\bigl(\T_n\times \T\rtimes \Z_n\bigr)^{(2)}\to
  \T_n\times \T\times n\Z$ of~\eqref{eq:phitilde} is defined by
  \[
    \tilde{\phi}\bigl((\chi,w,\dotk_1),(\chi,w\cdot
    \dotk_1,\dotk_2)\bigr)=\bigl(\chi,w,\omega(\dotk_1,\dotk_2)\bigr).
  \]
  By Lemma \ref{lem:push2coc}, $\Sigmaw$ is properly isomorphic to the
  extension by the 2-cocycle $\phiw$ which is the pushout of
  $\tilde{\phi}$.  Therefore
  $\phiw:\bigl(\T_n\times \T\rtimes \Z_n\bigr)^{(2)}\to \T_n\times
  \T\times \T$ is defined by 
  \[
    \phiw((\chi,w,\dotk_1),(\chi,w\cdot
    \dotk_1,\dotk_2))=(\chi,w,{\chi^{\omega(\dotk_1,\dotk_2)}}).
  \]
  Hence the rotation algebra $\A_\alpha$ is isomorphic to
  $C^*(\T_n\times\T\rtimes \Z_n;\Sigma_{\phiw} )$. For $\chi \in \T_n$, define
  $\chi_*(\phi) : (\T\rtimes Z_n)^{(2)} \to \T$ by
  \[
  {\chi}_*(\phi)((w,\dotk_1),(w\cdot
  \dotk_1,\dotk_2)=(w,{\chi}^{\omega(\dotk_1,\dotk_2)}).
  \]
  Then Proposition~\ref{prop:bundle} implies that $\A_\alpha$ is the
  section algebra of an upper-semi\-con\-tin\-uous $C^*$-bundle over
  $\T_n$ with fiber at $\chi\in \T_n$ isomorphic to
  $C^*(\T\rtimes Z_n;\Sigma_{{\chi}_*(\phi)})$.
\end{example}
%%%%%%%%%%%%%%%
%%%%% APPENDIX

\appendix

\section{Bundles of Twists}
\label{sec:bundles-twists}

Let $\Sigma$ be a twist over $\G$.   Alternatively,
$\Sigma$ is a $\T$-groupoid so that we have the following
diagram 
\begin{equation}
  \label{eq:11}
  \begin{tikzcd}[column sep=3cm]
      \goo\times\T \arrow[r,"i"] \arrow[dr,shift left, bend right = 15]
      \arrow[dr,shift right, bend right = 15]&\Sigma\arrow[r,"j",
      two heads] \arrow[d,shift left] \arrow[d,shift
      right]&\G, \arrow[dl,shift left, bend left = 15]
      \arrow[dl,shift right, bend left = 15]
      \\
      &\goo&
    \end{tikzcd}
  \end{equation}
  where as usual we have identified $\go$ and $\goo$.  In particular,
  if $F\subset \goo$ is $\G$-invariant, then it is $\Sigma$-invariant
  and the reduction $\Sigma\restr F$ is also a twist over the
  reduction $\G\restr F$.

  Suppose that $p:\goo\to T$ is a continuous map such that $p\circ
  r=r\circ s$.  Then we say that $\Sigma$ is a groupoid bundle over
  $T$.\footnote{The third author defined groupoid bundles in
    \cite{ren:ms15}*{Definition~3.3} where it is also required that $p$
    be open.}   Then $p^{-1}(t)$ is invariant for all $t\in T$.  We
  write $\Sigma(t)$ and $\G(t)$ for the restrictions to $p^{-1}(t)$,
  respectively.  Then $\Sigma(t)$ is a twist over $\G(t)$.

  \begin{prop}
    \label{prop-jean2} Suppose that $\G$ is a second countable locally
    compact Hausdorff groupoid with a Haar system and that $\Sigma$
    is a twist over $\G$.  If $p:\goo\to T$ is a continuous
    map such that $p\circ r=p\circ s$, then $\cs(\G;\Sigma)$ is a
    $C_{0}(T)$-algebra.  Let $\Sigma(t)$ be the
    twist over $\G(t)$ defined above. Then $\cs(\G;\Sigma)$ is
    (isomorphic to) the section algebra of an upper-semicontinuous
    \cs-bundle over $T$.  The fibre $\cs(\G;\Sigma)(t)$ is isomorphic
    to $\cs(\G(t);\Sigma(t))$.
  \end{prop}

  \begin{proof}
    Recall that $\cs(\G;\Sigma)$ is the \cs-algebra
    $\cs(\G,\B)$ of
    a Fell bundle $q:\B\to \G$ as described in
    \cite{muhwil:dm08}*{Example~2.9}.  Similarly,
    $\cs(\G(t);\Sigma(t))$ is the \cs-algebra $\cs(\G(t),\B)$ of
    $q\restr{q^{-1}(\G(t)) }$.  Let $U(t)=\goo\setminus p^{-1}(t)$.
    Using \cite{ionwil:hjm11}*{Theorem~3.7} (as in
    \cite{simwil:ijm13}*{Lemma~9}), we obtain a short exact sequence
    \begin{equation}
      \label{eq:12}
      \begin{tikzcd}
        0\arrow[r]&\cs(\G\restr{U(t)},\B)\arrow[r,"i"]&\cs(\G,\B)\arrow[r,"j"] &
        \cs(\G(t),\B) \arrow[r]&0
      \end{tikzcd}
    \end{equation}
    where $i$ identifies $\cs(\G\restr{U(t)},\B)$ with the completion
    in $\cs(\G,\B)$ of the
    ideal of sections in $\Gamma_{c}(\G,\B)$ that vanish off
    $\G\restr{U(t)}$, and $j$ is given on $\Gamma_{c}(\G,\B)$ by restriction
    to $p^{-1}(t)$.  Now 
    exactly as in \cite{wil:toolkit}*{Proposition~5.37}, we see that
    $\cs(\G,\B)$ is a $C_{0}(T)$-algebra with fibres $\cs(\G,\B)(t)$
    identified with $\cs(\G(t),\B)$.
  \end{proof}

  %%%%%%%%%%%%%%%%%%%%%%%%%%
  %%%%%  End Matter

\bibliographystyle{amsplain} \bibliography{iksw}
\end{document}